\documentclass[a4paper,10pt]{article}
\usepackage[utf8x]{inputenc}
\usepackage{amsfonts}
\usepackage{amssymb}
\usepackage{amsmath}
\usepackage{amsthm}
\usepackage{mathrsfs}               
\usepackage{bbm}               
\usepackage{array}
\usepackage{graphicx}
\usepackage{needspace}

\usepackage{fullpage}

\usepackage[numbers]{natbib}
\numberwithin{equation}{section}    

\numberwithin{equation}{section}

\newtheorem{satz}{Satz}[section]

\newtheorem{prop}[satz]{Proposition}
\newtheorem{theorem}[satz]{Theorem}
\newtheorem{proposition}[satz]{Proposition}
\newtheorem{corollary}[satz]{Corollary}
\newtheorem{lemma}[satz]{Lemma}

\theoremstyle{definition}
\newtheorem{definition}[satz]{Definition}

\newtheorem{remark}[satz]{Remark}

\newtheorem{example}[satz]{Example}

\newtheorem{assumption}{Assumption}

\DeclareMathOperator{\E}{{\mathbb E}}
\DeclareMathOperator{\F}{{\mathcal F}}

\DeclareMathOperator{\R}{{\mathbb R}}

\DeclareMathOperator{\N}{{\mathbb N}}

\DeclareMathOperator{\supp}{supp}
\DeclareMathOperator{\ran}{ran}
\DeclareMathOperator{\lin}{lin}

\DeclareMathOperator{\dom}{dom}

\DeclareMathOperator{\Id}{Id}

\renewcommand{\d}{\ensuremath {\,\mathrm{d}}}
\providecommand{\eps}{\varepsilon}
\providecommand{\ind}{\mathbbm{1}}
\providecommand{\scapro}[2]{\langle #1,#2 \rangle}
\renewcommand{\phi}{\varphi}
\renewcommand{\theta}{\vartheta}
\renewcommand{\subset}{\subseteq}
\renewcommand{\supset}{\supseteq}
\renewcommand{\cdot}{{\scriptstyle \bullet} }

\renewcommand{\le}{\leqslant}
\renewcommand{\ge}{\geqslant}
\renewcommand{\hat}{\widehat}
\renewcommand{\tilde}{\widetilde}

\title{Information bounds for inverse problems with application to deconvolution and L\'evy models}
\author{Mathias Trabs\footnote{E-mail address: trabs@math.hu-berlin.de}~\footnote{The author thanks Markus Rei{\ss} for many comments and stimulating questions and Nicolas Perkowski, Richard Nickl and Jakob S\"ohl for helpful remarks. This research was partially supported by the Deutsche Forschungsgemeinschaft through the FOR 1735 ``Structural Inference in Statistics''.}}

\date{Humboldt-Universit\"at zu Berlin}

\begin{document}

\maketitle

\begin{abstract}
If a functional in a nonparametric inverse problem can be estimated with parametric rate, then the minimax rate gives no information about the ill-posedness of the problem. To have a more precise lower bound, we study semiparametric efficiency in the sense of H\'ajek--Le Cam for functional estimation in regular indirect models. These are characterized as models that can be locally approximated by a linear white noise model that is described by the generalized score operator. 
A convolution theorem for regular indirect models is proved. This applies to a large class of statistical inverse problems, which is illustrated for the prototypical white noise and deconvolution model. It is especially useful for nonlinear models. We discuss in detail a nonlinear model of deconvolution type where a L\'evy process is observed at low frequency, concluding an information bound for the estimation of linear functionals of the jump measure. 
\end{abstract}



\section{Introduction}
Inverse problems are a key topic in applied mathematics, in particular models with noise in the data. Typically the parameter which is the target of the statistical inference is not directly observable, but ``hidden'' by some operator. While upper bounds, like convergence rates for nonparametric inverse problems, are mainly properties of the estimators, lower bounds reveal the deeper information theoretic structure. 
Instead of the (infinite dimensional) parameter itself, derived quantities are often the final object of interest. On the one hand, they might allow for inference with parametric rate, circumventing typical problems in nonparametric estimation like the choice of the bandwidth, cf. estimating the distribution function instead of the density. In this case minimax convergence rates give no information about the ill-posedness of the problem and we need the much more precise information bounds. On the other hand, many nonparametric statistical procedures rely on basis expansions and model selection strategies, see e.g. \citet{cavalierEtAl2002}. For these adaptive methods it is strictly necessary to assess the quality of the estimated coefficient in terms of confidence.

While inverse problems appear in many different shapes in the literature, information bounds are studied only in a few linear cases: \citet{klaassenEtAl2001, klaassenEtAl2005} and \citet{khoujmaneEtAl2007} considered the linear indirect regression model
\[
  Y_i= (K\theta)(X_i)+\xi_i,\quad i=1,\dots,n,
\]
where the regression function depends on the unknown parameter $\theta$ via the linear operator $K$ and $(X_1,Y_1),\dots,(X_n,Y_n)$ are observed in presence of the additional errors $\xi_i$. \citet{vanRooijEtAll1999} derived a convolution theorem for linear indirect density estimation, where a sample of i.i.d. random variables $Y_1,\dots,Y_n$ with distribution $K\theta$ is observed. If more specifically 
\[
  Y_i=X_i+\eps_i,\quad i=1,\dots,n,      
\]
where $X_i$ has law $\theta$ and is corrupted by the noise variable $\eps_i$, $K$ is a convolution operator. Efficiency for this so called deconvolution model was considered by \citet{soehlTrabs2012}. 

Using the polar decomposition or specific properties of the operators, all these studies are restricted to linear models. However, in many situations the operator $K$ might not be linear, see e.g. \citet{englEtAl1996} and \citet{bissantzEtAl2004}. Hence, new mathematical methods are necessary. The aim of the present paper is twofold: (a) to provide a convolution theorem for general inverse problems that are regular in a well specified sense and (b) to study concrete and prototypical examples of linear and nonlinear structure. 

A canonical probabilistic and nonlinear inverse problem is the following. Let $Y_i$ be compound Poisson distributed
\[
  Y_i\sim e^{-\lambda}\sum_{k=0}^\infty\frac{\lambda^k}{k!}\theta^{\ast k}
\]
with intensity $\lambda>0$ and jump distribution $\theta$, writing $\theta^{\ast k}$ for the $k$-fold convolution of the measure $\theta$. The distribution of $Y_i$ is a convolution exponential and therefore not linear in $\theta$. If $Y_i$ is more generally an increment of a L\'evy process $(L_t)_{t\ge0}$, inference on the characteristic triplet of the L\'evy process is a nonlinear problem since the dependence of the probability distribution of the marginals on the L\'evy triplet is determined by the characteristic exponent, see the review by \citet{reiss2013}. At the same time this model is of practical importance since L\'evy processes are the main building blocks for mathematical modeling of stochastic processes. In the related context of diffusion processes, efficient estimation was recently studied by \citet{clementEtAl2013}.

In view of the equivalence results by \citet{brownLow1996} and \citet{nussbaum1996}, the prototype of an inverse problem is to estimate $\theta\in\Theta\subset X$, or derived parameters, from observations $y_{\eps,\theta}$ in the white noise model
\begin{equation}\label{eqWhiteNoise}
  y_{\eps,\theta}=K(\theta)+\eps\dot W\quad\text{for a continuous operator }K:X\supset \Theta\to Y,
\end{equation}
where $X$ and $Y$ are Hilbert spaces and $\eps\dot W$ denotes white noise on $Y$ with noise level $\eps>0$. For a review of estimation results in this model we refer to \citet{cavalier2008} and references therein. Studying minimax convergence rates when $K$ is linear, \citet{goldenshlugerPereverzev2000, goldenshlugerPereverzev2003} have shown that the parametric rate $\eps$ can be achieved for linear functionals of $\theta$ whose smoothness is not smaller than the ill-posedness of the operator $K$. 

Inspired by the results by \citet{vanderVaart1991}, we restate the classical local asymptotic normality (LAN) theory in a way that is appropriate to capture the inverse structure of the above mentioned models. Here, the linear white noise model \eqref{eqWhiteNoise} serves as the local limit experiment in the sense of \citet{leCam1972}. This leads to the notion of \textit{regular indirect models}, meaning that the white noise model is the locally linear weak approximation of the statistical experiment. The asymptotic linear structure is described by the so called \textit{generalized score operator}. We derive a version of the H\'ajek--Le Cam convolution theorem for the estimation of derived parameters for regular inverse problems. The tangent set is determined by the range of the generalized score operator and the efficient influence function is given by the Moore--Penrose pseudoinverse of the adjoint score operator. Although we focus on linear functionals in the 
examples, the theory applies to any parameter which is differentiable in a pathwise sense. 

We show that the white noise model with a (possibly) nonlinear operator, the deconvolution model as well as the L\'evy model are regular indirect models and thus the convolution theorem applies. In many cases estimators are known that have the optimal limit distribution and consequently the information bound is sharp. The analysis of L\'evy processes is most challenging and the second half of this article is devoted to this model. Here, the proofs rely on estimates of the distance of infinitely divisible distributions by \citet{Liese1987} and the Fourier multiplier approach which was introduced by \citet{nicklReiss2012}.

We will put some stress on the L\'evy model for three reasons: First, it is an important paradigm for nonlinear problems in indirect density estimation. Second, to understand from an efficiency perspective the \textit{auto-deconvolution} structure of the L\'evy model which was first reported by \citet{belomestnyReiss2006}. Third, to answer a conjecture by \citet{nicklReiss2012}. Based on low frequency observations of a L\'evy process, they have constructed an estimator for the (generalized) distribution function of the jump measure $\nu$ and proved asymptotic normality when the parametric rate can be attained. 
The natural question is whether this estimator is efficient in the H\'ajek--Le Cam sense. Since \citet{buchmannGrubel2003} have constructed for a finite and known jump activity a decompounding estimator with smaller asymptotic variance, an information bound is of particular interest. 
With the general convolution theorem at hand, we can prove that both estimators are indeed efficient and thus prior knowledge of the jump intensity simplifies the statistical problem significantly. Concerning the information bound in the deconvolution setting, we can relax the assumptions on the functionals and the admissible error densities by \citet{vanRooijEtAll1999} and the assumptions on the smoothness and decay behavior of the densities of $X_i$ and $\eps_i$ by \citet{soehlTrabs2012} substantially. In fact, our abstract approach leads to natural assumptions in the explicit models.

This paper is organized as follows: Starting with the linear white noise model, we develop our general results in Section~\ref{secLAN}. These are illustrated in the deconvolution setup in Section~\ref{secDecon}. The theory will be applied to the L\'evy model in Section~\ref{secLevy}. While the previous sections are restricted to $\R^d$-valued functionals, we discuss the extension to general derived parameters in Section~\ref{secFunctionals}. More technical proofs are postponed to Section~\ref{secProofs}.

\section{Regular indirect models}\label{secLAN}
\subsection{Linear white noise model}
To understand the probabilistic structure of general inverse problems, we start with studying the abstract linear white noise model \eqref{eqWhiteNoise}, where $X$ and $Y$ are separable real Hilbert spaces with scalar products $\scapro\cdot\cdot_X$ and $\scapro\cdot\cdot_Y$, respectively, and $K:X\to Y$ is a linear and bounded operator. To avoid identifiability problems, we additionally assume that $K$ is injective. That is we observe for some unknown $\theta\in X$
\[
  \scapro {y_{\eps,\theta}}\phi_Y=\scapro{K\theta}{\phi}_Y+\eps\dot W(\phi)\quad\text{for all}\quad \phi\in Y
\]
where $(\dot W(\phi))_{\phi\in Y}$ is an iso-normal Gaussian process with mean zero and covariance structure $\E[\dot W(\phi_1)\dot W(\phi_2)]=\scapro{\phi_1}{\phi_2}_Y$ for $\phi_1,\phi_2\in Y$. The law $\mu$ of the white noise $\dot W$ is defined as symmetric (zero mean) Gaussian measure on $(E,\mathscr B(E))$ for a separable Banach space $E$ in which $Y$ can be continuously embedded and where $\mathscr B(E)$ denotes the Borel $\sigma$-algebra on $E$. In other words $\dot W$ is an isometry from $Y$ into $L^2(E,\mathscr B(E),\mu)$. For the construction of the so called abstract Wiener space we refer to \citet[Thm. 4.1, Lem. 4.7]{kuo1975}. We denote the law of $y_{\eps,\theta}$ by $P_{\eps,\theta}$.

Basically, the linear white noise model is a Gaussian shift experiment where the parameter is hidden behind the operator $K$. The inverse problem is to estimate a derived parameter $\chi(\theta)$ from the observation $y_{\eps,\theta}$ when $\eps\to0$. First, let us focus on a linear functional $\chi(\theta)=\scapro\zeta\theta_X$ for some $\zeta\in X$. Typically, $K$ is injective but admits no continuous inverse, leading to an ill-posed problem, cf. \citet{goldenshlugerPereverzev2000, goldenshlugerPereverzev2003} or \citet{cavalier2008} for a recent review of nonparametric estimation.

Following the classical semiparametric approach, we study parametric submodels by perturbing the parameter $\theta$ in directions $b\in X$. For any $b\in X$ we consider the submodel $t\mapsto P_{\eps,\theta_t}$ generated by the path $[0,1)\ni t\to\theta_t:=\theta+tb$. The behavior of the submodel along this path is described by the following lemma.
\begin{lemma}\label{lemDensity}
  Let $P_{\eps,x}$ denote the law of $y_{\eps,x}=K(x)+\eps\dot W$ on $(E,\mathscr B(E))$ for $x\in X$, and an operator $K:X\to Y$ with $K(0)=0$, then for all $\theta\in X$
  \[
    \frac{\d P_{\eps,x}}{\d P_{\eps,\theta}}(y_{\eps,\theta})=\exp\Big(\dot W\Big(\frac{K(x)-K(\theta)}{\eps}\Big)-\frac{1}{2\eps^2}\|K(x)-K(\theta)\|_Y^2\Big)\quad P_{\eps,\theta}-a.s.
  \]
\end{lemma}
The proof of this lemma relies on the Cameron--Martin formula for Gaussian measures on Banach spaces \citep[Prop. 2.24]{daPratoZabczyk1992} and is postponed to Section~\ref{secProofDens}. 
Linearity of $K$ yields $\eps^{-1}(K(\theta_{\eps})-K(\theta))=Kb$ and thus by Lemma~\ref{lemDensity}
\begin{equation}\label{eqCameronMartin}
  \log\frac{\d P_{\eps,\theta_{\eps}}}{\d P_{\eps,\theta}}(y_{\eps,\theta})
  =\dot W(Kb)-\frac{1}{2}\|Kb\|_Y^2\quad P_{\eps,\theta}-a.s.
\end{equation}
and therefore model \eqref{eqWhiteNoise} with linear operator $K$ satisfies the classical LAN condition (even without local and asymptotic) with parameter $h=Kb\in\ran K$. To find an information bound for the derived parameter $\chi(\theta)=\scapro\zeta\theta_X$, we express it in terms of $\eta=K\theta$ by
\[
  \psi(\eta):=\scapro{\zeta}{K^{-1}\eta}_X=\chi(\theta).
\]
Since $K^{-1}$ is typically not continuous, $\psi$ will not be continuous along the path $t\mapsto K\theta_t=\eta+tKb$ without further assumptions. Supposing however $\zeta\in \ran K^\star$, where $\ran K^\star$ denotes the range of the adjoint operator $K^\star$, continuity and linearity of $\psi$ follows from
\[
  \psi(\eta)=\scapro{K^\star y}{K^{-1}\eta}_X=\scapro y\eta_Y\quad\text{for any }y\in (K^\star)^{-1}(\{\zeta\}).
\]
In fact, we will see below that the condition $\zeta\in \ran K^\star$ is equivalent to the regularity of $\psi$. If $K^\star$ is not injective there are many solutions $y$ of the equation $K^\star y=\zeta$. The unique solution with minimal norm is given by the Moore--Penrose pseudoinverse 
\begin{equation}\label{eqMoorePenrose}
  (K^\star)^\dagger\zeta
  :=(K^\star|_{(\ker K^\star)^\perp})^{-1}(\zeta)
  =\Pi_{(\ker K^\star)^\perp}(K^\star)^{-1}(\{\zeta\})\quad\text{for }\quad\zeta\in\ran K^\star,
\end{equation}
where $\Pi_{(\ker K^\star)^\perp}$ denotes the orthogonal projection onto the orthogonal complement $(\ker K^\star)^\perp$ of the kernel of $K^\star$, cf. \citet[Def. 2.2 and Prop. 2.3]{englEtAl1996} for the definition and fundamental properties of the pseudoinverse.

Given the regularity of the parameter, an LAN version of the H\'ajek--Le Cam convolution theorem \citep[see][Thm. 3.11.2]{vanderVaartWellner1996} yields that the variance of any regular estimator is bounded from below by
\begin{equation}\label{eqLinInfoBound}
  \|\Pi_{\overline\ran K}(K^\star)^{-1}(\{\zeta\})\|_Y^2=\|(K^\star)^{\dagger}\zeta\|_Y^2\quad\text{if}\quad\zeta\in\ran K^\star,
\end{equation}
since the closure of the range of $K$ satisfies $\overline\ran K=(\ker K^\star)^\perp$. 
\begin{remark}
  Suppose that the operator $K$ is injective and compact and denote the domain of $K$ by $\dom K$. Then $K$ is adapted to the Hilbert scale $(\dom (K^\star K)^{-\alpha})_{\alpha\ge0}$ generated by $(K^\star K)^{-1}$ and its degree of ill-posedness is $\alpha=1/2$ \cite[cf.][]{natterer1984}. According to \cite{goldenshlugerPereverzev2000}, the parameter $\chi(\theta)=\scapro\zeta\theta_X$ can be estimated with parametric rate $\eps$ if and only if $\theta\in\ran((K^\star K)^{1/2})$. Noting that $\ran K^\star=\ran((K^\star K)^{1/2})$ \citep[Prop. 2.18]{englEtAl1996}, we recover the condition $\zeta\in \ran K^\star$. Since the existence of a regular estimator implies in particular that $\chi(\theta)$ can be estimated with rate $\eps$, this condition is natural for stating a convolution theorem.
\end{remark}
\begin{remark}\label{remUpperBound}
  The information bound in \eqref{eqLinInfoBound} is sharp. Usually, regularization methods are necessary to construct estimators in ill-posed problems because $K^\dagger$ is unbounded and the observation $y_{\eps,\theta}$ may not be in its domain. Assuming $\zeta\in\ran K^\star$, we can however define the estimator $\hat\chi(\theta):=\scapro{(K^\star)^\dagger\zeta}{y_{\eps,\theta}}_Y$ which satisfies
  \begin{align*}
    \hat\chi(\theta)-\chi(\zeta)
    &=\scapro{(K^\dagger)^\star \zeta}{K\theta}_Y-\scapro\zeta\theta_X+\eps\dot W((K^\dagger)^\star\zeta)\\
    &=\scapro\zeta{(K^\dagger K-\Id)\theta}_X+\eps\dot W((K^\dagger)^\star\zeta)
    \sim\mathcal N(0,\eps^2\|(K^\dagger)^\star\zeta\|_Y^2)
  \end{align*}
  where we used $K^\dagger K=\Pi_{(\ker K)^\perp}=\Id$ because $K$ is assumed to be injective. Therefore, the estimator $\hat\chi(\theta)$ is efficient.
\end{remark}

If the operator $K$ in model \eqref{eqWhiteNoise} is nonlinear, the situation is more involved and a naive approach may fail as the following example illustrates. We define the Fourier transform of a function $f\in L^1(\R)\cup L^2(\R)$ as $\F f(u):=\int e^{iux}f(x)\d x$.
\begin{example}\label{exDiffEq}
  For a given $\theta\in \Theta:=\{f\in L^2(\R)|f\ge0\}\subset X:=L^2(\R)$ consider the linear differential equation in $f$
  \begin{equation}\label{eqDiffEq}
    f'=-f+\theta^2\quad\text{with}\quad \lim_{t\to-\infty}f(t)=0,
  \end{equation}
  which has the explicit solution $f_\theta(t):=\int_{-\infty}^te^{-(t-s)}\theta^2(s)\d s$.
  The inverse problem is to estimate a linear functional $\chi(\theta)=\scapro\theta\zeta_X, \zeta\in X,$ given an observation of the solution $f_\theta\in Y:=L^2(\R)$ of the previous equation corrupted by white noise. Since \eqref{eqDiffEq} is equivalent to $\F[\theta^2]=\F[f+f']=(1-iu)\F f$, the operator $K:\Theta\to L^2(\R)$, which maps $\theta$ to the solution $f_\theta$, can be written as
  \[
    K(\theta)=\F^{-1}\big[(1-iu)^{-1}\F[\theta^2](u)\big].
  \]
  Note that $K$ is well defined because $\|(1-i\cdot)^{-1}\F[\theta^2]\|_{L^2}\le\|(1-i\cdot)^{-1}\|_{L^2(\R)}\|\theta\|_{L^2(\R)}^2$. We see immediately that $K$ is nonlinear and injective on $\Theta$. Due to the derivative in \eqref{eqDiffEq}, $\theta$ does not depend continuously on the data $f_\theta$ and thus the problem is ill-posed. 
  
  Following the strategy of the linear model, we introduce the direct parameter $\eta=K(\theta)$ and write $\psi(\eta)=\scapro{\zeta}{K^{-1}(\eta)}_X=\chi(\theta)$. Note that $\psi$ is nonlinear in $\eta$. To study pathwise continuity of $\psi$, we consider the path $[0,1)\ni t\mapsto\eta_t=\eta+th$ with direction $h=K(b), b\in \Theta$. Note that $\eta_t\in\ran K$ since 
  \[
    \big(\eta_t+\eta_t'\big)^{1/2}=\big(\eta+\eta'+t(h+h')\big)^{1/2}=\big(\theta^2+tb^2\big)^{1/2}\in \Theta.
  \]
  For some intermediate point $\xi\in[0,t]$ the mean value theorem yields
  \begin{align}
   &t^{-1}\big(\psi(\eta_t)-\psi(\eta)\big)
   =t^{-1}\scapro{K^{-1}(\eta_t)-K^{-1}(\eta)}\zeta_X\notag\\
   =&\scapro{\tfrac{1}{2}(\eta+\eta')^{-1/2}(h+h')}\zeta_X-t\scapro{\tfrac{1}{4}(\eta_\xi+\eta_\xi')^{-3/2}(h+h')^2}\zeta_X.\label{eqExLinearization}
  \end{align}
  The first term is the linearization
  $
    \dot\psi_\eta(h)=\frac{1}{2}\scapro{(\eta+\eta')^{-1/2}(h+h')}{\zeta}_X
    =\frac{1}{2}\scapro{\theta^{-1}b^2}{\zeta}_X
  $
  where we have to impose suitable conditions on $\zeta$ first to compensate the potentially non-integrable singularities of $\theta^{-1}$ and second to ensure continuity in $h$. But even if these conditions are satisfied, pathwise continuity of $\psi$ may fail because the integrability problems in the remainder in \eqref{eqExLinearization} are more serious because $b^4$ is not integrable for every $b\in X$ and the singularities of $(\theta^2+\xi b^2)^{-3/2}$ are more restrictive. 
\end{example}

What went wrong in Example~\ref{exDiffEq}? Regularity of the parameter $\psi$ depends on two properties: (i) the choice of $\zeta$ and (ii) the directions and paths along which we want to show the regularity. In particular the second point has to capture the inverse structure of the problem. The approach in the following section provides a solution to both problems. It gives a clear condition on $\zeta$ and it determines appropriate perturbations of the parameter, described by the tangent space.

\subsection{Local linear weak approximation}
Turning to a much more general model, the following definition will ensure that it behaves locally like the model \eqref{eqWhiteNoise} with a linear operator. Let $\Theta$ be a parameter set such that for any $\theta\in\Theta$ there is a \textit{tangent set} $\dot\Theta_\theta$ that is a subset of a Hilbert space with scalar product $\scapro\cdot\cdot_\theta$ such that any element $b\in\dot\Theta_\theta$ is associated to a path $[0,\tau)\ni t\mapsto\theta_t\in\Theta$ starting at $\theta$ and for some $\tau>0$. For the sake of brevity we suppress the dependence of the path on $b$ in the notation. In the following $Y_n\overset{P_n}{\Rightarrow}Y$ denotes weak convergence of the law of $Y_n$ under the measure $P_n$ to the law of $Y$ for random variables $Y_1,Y_2,\dots$ and $Y$.
\begin{definition}\label{defReg}
   The sequence of statistical experiments $(\mathcal X_n,\mathcal A_n,P_{n,\theta}:\theta\in\Theta)$ is called a \textit{locally regular indirect model} at $\theta\in\Theta$ with respect to the tangent set $\dot\Theta_\theta$ if there are a Hilbert space $(H_\theta,\scapro \cdot\cdot_H)$ and a continuous linear operator 
  \[
    A_\theta:\lin \dot \Theta_\theta\to H_\theta
  \]
  such that for some rate $r_n\downarrow0$ and for every $b\in\dot\Theta_\theta$ with associated path $t\mapsto \theta_t$ there are random variables $(G_n(h))_{h\in\ran A_\theta}$ satisfying
  \begin{gather}\label{eqLAN}
    \log\frac{\d P_{n,\theta_{r_n}}}{\d P_{n,\theta}}=G_n(A_\theta b)-\frac{1}{2}\|A_\theta b\|_H^2 \quad\text{and}\\
    \big(G_n(h_1),\dots,G_n(h_k)\big)\overset{P_{n,\theta}}{\Longrightarrow}\big(G(h_1),\dots,G(h_k)\big)\text{ for all } k\in\N,h_1,\dots h_k\in\ran A_\theta\notag
  \end{gather}
  for a centered Gaussian process $(G(h))_{h\in\ran A_\theta}$ with covariance $\E[G(h_1)G(h_2)]=\scapro{h_1}{h_2}_H$. The operator $A_\theta$ is called \textit{generalized score operator}.
\end{definition}
In the sequel we will use the notation
\[
  \mathcal P_n:=\{P_{n,\theta}|\theta\in\Theta\}.
\]
The statistical interpretation of this regularity becomes clear by comparing it to the likelihood ratio \eqref{eqCameronMartin} in the linear white noise model. Condition~\eqref{eqLAN} means that locally at $\theta$ the model $(P_{n,\theta_{r_n}})$ converges to a limit experiment which is a linear inverse problem \eqref{eqWhiteNoise} in white noise with operator $K=A_\theta$ on the Hilbert space $H_\theta$ and with noise level $\eps_n=r_n$. In other words at $\theta$, the model converges weakly to the linear inverse problem in the sense of \citet{leCam1972}. Therefore, the classical white noise model~\eqref{eqWhiteNoise} serves as a \textit{locally linear weak approximation} of the general model $\mathcal P_n$. The difference to the classical theory is that the limit experiment is not a direct Gaussian shift experiment, but an indirect Gaussian shift, preserving the inverse structure of the problem. In that sense property \eqref{eqLAN} generalizes the classical local asymptotic normality, 
which corresponds to the identity operator $A_\theta=\Id$, to \textit{local asymptotic indirect normality (LAIN)}. 

The derived parameter $\chi:\Theta\to\R^d$, which is the aim of the statistical inference, should then be regular in the following sense.
\begin{definition}
 The function $\chi:\Theta\mapsto \R^d, d\in\N$, is \textit{pathwise differentiable} at $\theta\in\Theta$ with respect to the tangent set $\dot\Theta_\theta$ if there is a continuous linear operator $\dot\chi_\theta:\lin \dot\Theta_\theta\to\R^d$ such that for every $b\in\dot\Theta_\theta$ with associated path $[0,\tau)\ni t\mapsto \theta_t\in\Theta$ it holds
 \[
   \frac{1}{t}\big(\chi(\theta_{t})-\chi(\theta)\big)\to\dot\chi_\theta b\quad\text{as }t\downarrow0.
 \]
\end{definition}
By the Riesz representation theorem we can write $\dot\chi_\theta b=\scapro{\tilde\chi_\theta}b_\theta$ for all $b\in\dot\Theta_\theta$ and some gradient $\tilde\chi_\theta\in\overline\lin\dot\Theta_\theta$. Recall that the sequence of parameter functions $\psi_n:\mathcal P_n\to\R^d$ given by $\psi_n(P_{n,\theta})=\chi(\theta)$ is called regular at $\theta$ relative to $A_\theta\dot\Theta_\theta$ if for any $h\in A_\theta\dot\Theta_\theta$ and any submodel $t\mapsto P_{n,\theta_t}$ satisfying \eqref{eqLAN} with $h=A_\theta b$ for some $b\in\dot\Theta_\theta$, it holds
\begin{equation}\label{eqDiffPsi}
  \frac{\psi_n(P_{n,\theta_{r_n}})-\psi_n(P_{n,\theta})}{r_n}\to\dot\psi_\theta(h)
\end{equation}
for some continuous linear map $\dot\psi_\theta:H\to\R^d$. Again the Riesz representation theorem determines a unique $\tilde\psi_\theta\in\overline{\ran }A_\theta =\overline{\lin}A_\theta\dot\Theta_\theta$ such that $\dot\psi_\theta(h)=\scapro{\tilde\psi_\theta}{h}_H$ for all $h\in\ran A_\theta$. $\tilde\psi_\theta$ is called \textit{efficient influence function} in the classical semiparametric theory. 
As the last ingredient we recall that a sequence of estimators $T_n:\mathcal X_n\to \R^d$ is called \textit{regular} at $\theta$ with respect to the rate $r_n$ and relative to the directions $\dot\Theta_\theta$ if there is a limit distribution $L$ on the Borel measurable space $(\R^d,\mathscr B(\R^d))$ such that
\[
  \frac{1}{r_n}\big(T_n-\chi(\theta_{r_n})\big)\overset{P_{n,\theta_{r_n}}}{\Longrightarrow} L
\]
for every $b\in\dot{\Theta}_\theta$ and any corresponding submodel $t\mapsto P_{n,\theta_t}$. We recall the definition \eqref{eqMoorePenrose} of the Moore--Penrose pseudoinverse $K^\dagger$ of an operator $K$ on its range and
obtain the following convolution theorem.
\begin{theorem}\label{thmScore}
  Let $(\mathcal X_n,\mathcal A_n,P_{n,\theta}:\theta\in\Theta)$ be a locally regular indirect model at $\theta\in\Theta$ and $\chi:\Theta\to\R^d$ be pathwise differentiable at $\theta$ with respect to $\dot\Theta_\theta$. Then the sequence $\psi_n:\mathcal P_n\to\R^d$ is regular at $\theta$ relative to $\dot\Theta_\theta$ if and only if each coordinate function of $\tilde\chi_\theta=(\tilde\chi_\theta^{(1)},\dots,\tilde\chi_\theta^{(d)})$ is contained in the range of the adjoint score operator $A_\theta^\star:H\to\overline{\lin}\dot\Theta_\theta$. \par
  In this case the efficient influence function is given by $\tilde\psi_{\theta}=(A_\theta^\star)^\dagger \tilde\chi_\theta=((A_\theta^\star)^\dagger \tilde\chi_\theta^{(1)},\dots,(A_\theta^\star)^\dagger \tilde\chi_\theta^{(d)})$ and for any regular estimator sequence $T_n:\mathcal X_n\to\R^d$ the limit distribution satisfies $L=\mathcal N(0,\Sigma)\ast M$ for some Borel probability distribution $M$ and covariance matrix $\Sigma\in\R^{d\times d}$ with
  \begin{equation}\label{eqInfoBound}
    \Sigma_{k,l}=\big\langle\tilde\psi_\theta^{(k)},\tilde\psi_\theta^{(l)}\big\rangle_H
    =\big\langle(A_\theta^\star)^\dagger \tilde\chi_\theta^{(k)},(A_\theta^\star)^\dagger \tilde\chi_\theta^{(l)}\big\rangle_H,\quad k,l\in\{1,\dots,d\}.
  \end{equation}
\end{theorem}
\begin{proof}
  The characterization of regular parameter functions can be proved analogously to the i.i.d. setting studied by \citet[Thm. 3.1]{vanderVaart1991}. Let us first consider $d=1$. For any $b\in\dot\Theta_\theta$ there is a path $[0,\tau)\ni t\mapsto \theta_t$ in direction $b$ generating a submodel $t\mapsto P_{n,\theta_t}$ with $h=A_\theta b$. If $\psi_n$ is regular,
  \[
    \dot\psi_\theta(A_\theta b)
    =\lim_{n\to\infty}\frac{\psi_n(P_{n,\theta_{r_n}})-\psi_n(P_{n,\theta})}{r_n}
    =\lim_{n\to\infty}\frac{\chi(\theta_{r_n})-\chi(\theta)}{r_n}
    =\scapro{\tilde\chi_\theta}b_\theta.
  \]
  Since the equality $\scapro{\tilde\psi_{\theta}}{A_\theta b}_H=\dot\psi_\theta(A_\theta b)=\scapro{\tilde\chi_\theta}b_\theta$ holds for all $b\in\dot\Theta_\theta$, it follows $A_\theta^\star \tilde\psi_\theta=\tilde\chi_\theta$. To conclude the converse direction, we can use the previous display as definition of $\dot\psi_\theta$ and have to verify that it is indeed linear and continuous. But this follows because by assumption there is  some $\psi\in H$ such that $A_\theta^\star \psi=\tilde\chi_\theta$ and thus $\dot\psi_\theta(h)=\scapro{\psi}{h}_H$ for all $h\in H$.  Because $\overline\ran A_\theta=(\ker A_\theta^\star)^\perp$, there is exactly one solution of $A_\theta^\star\psi=\tilde\chi_\theta^{}$ in $\overline\ran A_\theta$ and this is given by $(A_\theta^\star)^\dagger \tilde\chi_\theta=\Pi_{(\ker A_\theta^\star)^\perp}(A_\theta^\star)^{-1}(\{\tilde\chi_\theta\})$. For $d>1$ it is sufficient to consider the coordinate functions separately.

  To conclude the second part of the theorem, we consider $A_\theta\dot\Theta_\theta$ as local parameter set and identify the local parameter $\kappa_n(A_\theta b):=\psi_n(P_{n,\theta_{r_n}})=\chi(\theta_{r_n})$ with $\kappa_n(0):=\psi_n(P_{n,\theta})$. Then $\kappa_n$ is regular relative to $A_\theta\dot\Theta_\theta$ and we can apply the convolution theorem in \citet[Thm. 3.11.2]{vanderVaartWellner1996}. Hereby, we have to note that it is sufficient if the local parameter set $A_\theta\dot\Theta_\theta$ is only a subset of a Hilbert space, and thus \eqref{eqLAN} may not hold for all linear combinations of elements in $\dot\Theta_\theta$, as long as the weak convergence $G_n(h)\Rightarrow G(h)$ under $P_{n,\theta}$ holds true for all $h\in\lin A_\theta\dot\Theta_\theta$. 
\end{proof}
The theorem implies immediately that the asymptotic covariance of every regular estimator of $\chi(\theta)$ is bounded from below by \eqref{eqInfoBound} in the order of nonnegative definite matrices. 
If $\tilde\chi_\theta$ is contained in the smaller range of the \textit{information operator} $A_\theta^\star A_\theta$, then the efficient influence function can be obtained by 
\[
  \tilde\psi_\theta=(A_\theta^\star)^\dagger\tilde\chi_\theta=A_\theta(A_\theta^\star A_\theta)^\dagger\tilde\chi_\theta,
\]
owing to $\ker(A_\theta^\star A_\theta)=\ker A_\theta$. Therefore, in this case the hardest parametric subproblem is given by the direction $(A_\theta^\star A_\theta)^{\dagger}\tilde\chi_\theta\in \overline\lin\dot\Theta_\theta$. In the finite dimensional linear model this lower bound coincides with the minimal variance of the Gau\ss--Markov theorem. Let us illustrate Theorem~\ref{thmScore} in several examples.
\begin{example}[Indirect regression model]\label{exIndReg}
  With $X=Y=L^2(\R)$ and a linear, bounded, injective operator $K:X\to Y$ consider the indirect regression model with deterministic design
  \[
    Y_i=(Kf)\Big(\frac{i}{n}\Big)+\eps_i,\quad i=1,\dots,n,\quad\text{with unkown }f\in X
  \]
  and with i.i.d. errors $\eps_1,\dots,\eps_n\sim\mu$ for some law $\mu$. Therefore, $(Y_1,\dots,Y_n)\sim P_{n,f}=\Pi_{i=1}^n \mu(\cdot -(Kf)(i/n))$. \citet{khoujmaneEtAl2007} have proved a convolution theorem for estimating $\chi(f)=\scapro\zeta f_{L^2(\R)}, \zeta\in X$ with $\|\zeta\|_{L^2(\R)}=1$, where $f$ and the error distribution $\mu$ are unknown. Let us focus on the submodel with standard normal errors $\eps_i\sim \mathcal N(0,1)$. Under smoothness conditions it is stated in this submodel \citep[Thm. 2 with $\delta=0$]{khoujmaneEtAl2007} that the asymptotic distribution of any regular estimator is given by the convolution $\mathcal N(0,\|K\zeta\|_{L^2(\R)}^{-2})*M$ for some Borel probability measure $M$. 
  
  To apply Theorem~\ref{thmScore}, we have to check that this model is a regular indirect model. Choosing the tangent space $\dot\Theta_f=X$ with linear paths $f_t=f+tb$ in directions $b\in\dot\Theta_f$, we calculate
  \begin{align*}
    &\log \frac{\d P_{n,f_{1/\sqrt n}}}{\d P_{n,f}}(Y_1,\dots,Y_n)\\
    &\qquad=-\frac{1}{2n}\sum_{i=1}^n(Kb)^2\Big(\frac{i}{n}\Big)+\frac{1}{n^{1/2}}\sum_{i=1}^n(Kb)\Big(\frac{i}{n}\Big)\Big(Y_i-(Kf)\Big(\frac{i}{n}\Big)\Big)\\
    &\qquad\overset{P_{n,f}}\Longrightarrow\mathcal N\Big(-\frac{1}{2}\|Kb\|_{L^2(\R)}^2,\|Kb\|_{L^2(\R)}^2\Big).
  \end{align*}
  Therefore, the generalized score operator is given by $A_f=K$. Assuming for simplicity that $\zeta\in\ran (K^\star K)$, the asymptotic distribution of any regular estimator is given by a convolution $\mathcal N(0,\|K(K^\star K)^{\dagger}\zeta\|_{L^2(\R)}^2)\ast M$. Since the Cauchy--Schwarz inequality yields $\|K\zeta\|_{L^2(\R)}\|K(K^\star K)^{\dagger}\zeta\|_{L^2(\R)}\ge\|\zeta\|_{L^2(\R)}^2=1$, the bound by \cite{khoujmaneEtAl2007} achieves our information bound if and only if $K^\star K=\lambda \Id$ for some $\lambda>0$. Therefore, their information bound may not be optimal. The reason is that $f$ has been perturbed in direction $\zeta$ instead of the the least favorable direction $(K^\star K)^{\dagger}\zeta$.
\end{example}

\begin{example}[Nonlinear white noise model]\label{exWhiteNoise}
  Suppose we observe $y_{n,\theta}=K(\theta)+\eps_n \dot W$ with $\eps_n\to0$ as $n\to\infty$ on the Hilbert space $Y$ for some $\theta\in X$ and for a not necessarily linear operator $K:X\supset\Theta\to Y$ with $K(0)=0$ which is G\^{a}teaux differentiable at the inner point $\theta\in\Theta$. That is there is a continuous linear operator $\dot K_\theta:X\to Y$ with 
  \[
    \lim_{t\to0}\frac{1}{t}\big(K(\theta+tb)-K(\theta)\big)= \dot K_\theta b \quad\text{for all}\quad b\in X.
  \]
  
  By the Hilbert space structure, the tangent space can be chosen as $\dot\Theta_\theta=X$ by considering the path $[0,1)\ni t\mapsto\theta_t:=\theta+tb$ for $b\in\dot\Theta_\theta$. Lemma~\ref{lemDensity} yields for any $b\in X$ with associated path $t\mapsto\theta_t$
  \[
    \log \frac{\d P_{n,\theta_{\eps_n}}}{\d P_{n,\theta}}(y_{n,\theta})
    =\dot W\Big(\frac{K(\theta+\eps_nb)-K(\theta)}{\eps_n}\Big)-\frac{1}{2\eps_n^2}\|K(\theta+\eps_nb)-K(\theta)\|_Y^2.
  \]
  Therefore, the LAIN property~\eqref{eqLAN} is satisfied with generalized score operator chosen as the G\^ateaux derivative $A_\theta=\dot K_\theta$ at $\theta$ and 
  \begin{align*}
    G_n(A_\theta b)=\dot W\Big(\frac{K(\theta+\eps_nb)-K(\theta)}{\eps_n}\Big)-\frac{1}{2}\Big(\Big\|\frac{K(\theta+\eps_nb)-K(\theta)}{\eps_n}\Big\|_Y^2-\|\dot K_\theta b\|_Y^2\Big)
  \end{align*}
  where $G_n(A_\theta b)\Rightarrow \mathcal N(0,\|\dot K_\theta b\|_Y^2)$, since the variance of the first term of $G_n$ converge to $\|\dot K_\theta b\|_Y^2$ and the second term converges deterministically to zero by the G\^ateaux differentiability. The convergence of the finite dimensional distributions follows likewise.

  Along the path $t\mapsto\theta_t$ the linear functional $\chi(\theta)=\scapro\zeta \theta_X$ possesses the derivative
  \[
    \lim_{t\to0}\frac{1}{t}\big(\chi(\theta_t)-\chi(\theta)\big)=\scapro\zeta b_X.
  \] 
  Hence, $\dot\chi_\theta b=\scapro{\tilde\chi_\theta}b_X$ for the gradient $\tilde\chi_\theta=\zeta$ and all $b\in\dot\Theta_\theta$. 
  Applying Theorem~\ref{thmScore} shows in particular that the asymptotic variance of every regular sequence of estimators $T_n$ (with respect to the rate $\eps_n$) of the functional $\chi(\theta)$ is bounded from below by
  \[
    \|(\dot K_\theta^\star)^{\dagger}\zeta\|_Y^2\quad\text{whenever }\quad\zeta\in\ran\dot K_\theta^\star.
  \]
  If $K$ is a linear bounded operator the score operator is $A_\theta=\dot K_\theta=K$ and thus the statement of Theorem~\ref{thmScore} coincides with the previous result~\eqref{eqLinInfoBound}.

  In Remark~\ref{remUpperBound} we saw that that this information bound can be achieved if $K$ is linear. For nonlinear operators an upper bound is beyond the scope of this paper.
\end{example}
\begin{example}[Example~\ref{exDiffEq} continued]\label{exDiffEq2}
  Let us come back to the ill-posed inverse problem in Example~\ref{exDiffEq} related to the differential equation~\eqref{eqDiffEq}. The corresponding nonlinear operator $K:X\to Y$ with $\Theta=\{f\in L^2(\R)|f\ge0\}\subset X$ and $X=Y=L^2(\R)$ was given by $K(\theta)=\F^{-1}[(1-i\cdot)^{-1}\F[\theta^2]]$. For $\theta\in \Theta$ and any $b\in \dot\Theta_\theta:=\Theta$ the functional $\chi(\theta)=\scapro\zeta\theta_X,\zeta\in L^2(\R)$, is pathwise differentiable along the path $[0,1)\mapsto\theta_t=\theta+tb$ with gradient $\tilde\chi_\theta=\zeta$. $K$ is pathwise differentiable with respect to the tangent set $\dot\Theta_\theta$ at $\theta$ with derivative  
  \[
    \dot K_\theta b=\F^{-1}\big[(1-iu)^{-1}\F[2b\theta](u)\big].
  \]
  Since $\dot K$ is well defined on $\lin\dot\Theta_\theta=L^2(\R)$, the generalized score operator $A_\theta:\lin\dot\Theta_\theta\to H:=L^2(\R)$ is given by $A_\theta b=\dot K_\theta b$ as in the previous example. The ``directions'' in which we perturb the direct parameter $K(\theta)$ are then given by $A_\theta \dot\Theta_\Theta=\{K(\sqrt{\theta b})|b\in X\}$. Applying Plancherel's identity twice, the adjoint of $A_\theta$ can be calculated via
  \begin{align*}
    \scapro {\dot K_\theta b}h
    &=\frac{1}{2\pi}\int(1-iu)^{-1}\F[2b\theta](u)\overline{\F h}(u)\d u\\
    &=\scapro b{2\theta\F^{-1}\big[(1-iu)^{-1}\F h(-u)\big]}_X,
  \end{align*}
  for $b,h\in L^2(\R)$. Therefore, $A_\theta^\star h=2\theta\F^{-1}[(1-iu)^{-1}\F h(-u)]$ and regularity of the parameter function follows for any $\zeta\in\ran A_\theta^\star=\{\theta f|f\in H^1(\R)\}$ with the Sobolev space $H^1(\R)=\{f\in L^2(\R)|\int(1+u^2)|\F f(u)|^2\d u<\infty\}$.
\end{example}
As Example~\ref{exDiffEq2} indicates, the adjoint score operator $A_\theta^\star$ has usually no closed range. In these cases it is a difficult problem to determine the range of $A_\theta^\star$. As the following characterization shows, it is sufficient to know $A_\theta^\star$ on a dense subspace. This approximation argument will turn out to be very useful for more complex models.
\begin{prop}\label{propApprox}
  Let $A_\theta^\star:H\to\overline{\lin}\dot\Theta_\theta$ be injective and let $\mathcal G$ be a dense subspace in $H$. Then $\tilde\chi_\theta\in\overline{\lin}\dot\Theta_\theta$ is contained in $\ran A_\theta^\star$ if and only if the following is satisfied
  \begin{enumerate}
    \item there exists a sequence $\chi_n\in\ran A_\theta^\star|_{\mathcal G}$ such that $\chi_n\to\tilde\chi_\theta$ as $n\to\infty$ and
    \item $(A_\theta^\star)^{-1}\chi_n$ converges to some $\psi\in H$.
  \end{enumerate}
  In this case $A_\theta^\star\psi=\tilde\chi_\theta$ and thus $\Pi_{\overline{\ran }A_\theta}\psi=\psi$ is the efficient influence function.
\end{prop}
\begin{proof}
  \textit{``if'':}
  Since $A_\theta^\star$ is a bounded operator, its graph 
  \[
    \{(g,A_\theta^\star g):g\in H\}\subset H\times\overline{\lin}\dot\Theta_\theta   
  \]
  is closed. Therefore, the inverse operator $(A_\theta^\star)^{-1}|_{\ran A_\theta^\star}$ is closed, too. Consequently, (i) and (ii) imply $\tilde\chi_\theta\in\dom (A_\theta^\star)^{-1}=\ran A_\theta^\star$ with $(A_\theta^\star)^{-1}\tilde\chi_\theta=\psi$.

  \textit{``only if'':} Since $\tilde\chi_\theta\in\ran A_\theta^\star$ there is some $\psi\in H$ such that $A_\theta^\star\psi=\tilde\chi_\theta$. Moreover, there is a sequence $(g_n)\subset\mathcal G$ with $g_n\to\psi$ because $\mathcal G$ is dense in $H$. The continuity of $A_\theta^\star$ yields then $\chi_n:=A_\theta^\star g_n\to A_\theta^\star \psi=\tilde\chi_\theta$.
\end{proof}

\subsection{I.i.d. observations}\label{secIid}
When the observations are given by $n$ independent and identically distributed random variables $Y_1,\dots,Y_n$, the model simplifies to the product space $(\mathcal X^n,\mathcal A^{\otimes n},P_\theta^{\otimes n}:\theta\in\Theta)$ such that the probability measure is completely described by the family of marginal distributions $\mathcal P=\{P_\theta:\theta \in\Theta\}$. We will rephrase the conditions of the previous section in terms of the marginal measure $P_\theta$. This setting appears quite often in applications and, in particular, the deconvolution model and the L\'evy model which we study in Sections~\ref{secDecon} and \ref{secLevy} will be two examples. That is why, we will give some details for the i.i.d. case.

Recall that a \textit{tangent set} $\dot{\mathcal P}_{P_\theta}$ at $P_\theta$ is a set of score functions $g$ of submodels $[0,\tau)\ni t\mapsto P_{\theta_t}$ starting at $P_\theta$ and for some $\tau>0$. In the present situation the derived parameter can be written as $\psi(P_\theta)=\chi(\theta)$, independent of $n$. The classical Haj\'ek--Le Cam convolution theorem \citep[cf.][Thm. 3.3.2]{bicklEtAl1998} applies if $\psi$ is differentiable at $P_\theta$ relative to $\dot{\mathcal P}_{P_\theta}$, that is,  
there exists a continuous linear map $\dot\psi:L^2(P_\theta)\to\R^k$ such that
\[
  \lim_{t\to0}\frac{\psi(P_{\theta_t})-\psi(P_\theta)}{t}=\dot\psi g\quad\text{for all } g\in\dot{\mathcal P}_{P_\theta}.
\] 
This differentiability corresponds to the general assumption \eqref{eqDiffPsi}. In the i.i.d. setting local asymptotic normality follows from Hellinger regularity of the submodel $t\mapsto P_{\theta_t}$. Therefore, we can reformulate the conditions in Definition~\ref{defReg} to the following
\begin{assumption}\label{assRegularity}
  At $\theta\in\Theta$ let the parameter set give rise to a tangent set $\dot \Theta_\theta$. Furthermore, let there be a continuous linear operator 
  \[
    A_\theta:\lin \dot \Theta_\theta\to L_0^2(P_\theta):=\Big\{g\in L^2(P_\theta):\int g\d P_\theta=0\Big\}\quad
  \]
  such that for every $b\in\dot\Theta_\theta$ with associated path $t\mapsto\theta_t$
  \begin{equation}\label{eqHellingerDeriv}
    \int\Big(\frac{\d P_{\theta_t}^{1/2}-\d P_\theta^{1/2}}{t}-\frac{1}{2}A_\theta b\d P_\theta^{1/2}\Big)^2\to0\quad\text{as }t\downarrow0.
  \end{equation}
\end{assumption}
In \eqref{eqHellingerDeriv} $\d P_{\theta_t}$ denotes the Radon--Nikodym $\mu$-density of $P_{\theta_t}$ for some dominating measure $\mu$ and the integration is with respect to $\mu$. Since the integral does not depend on $\mu$, it is suppressed in the notation.
\begin{lemma}\label{lem:hellinger}
  If the product model $(\mathcal X^n,\mathcal A^{\otimes n},P_\theta^{\otimes n}:\theta\in\Theta)$ satisfies Assumption~\ref{assRegularity} at $\theta\in\Theta$, then it is a locally regular indirect model at $\theta\in\Theta$ with respect to the tangent set $\dot\Theta_\theta$, with rate $r_n=n^{-1/2}$ and (generalized) score operator $A_\theta$.
\end{lemma}
\begin{proof}
  The Hellinger regularity in Assumption~\ref{assRegularity} implies local asymptotic normality, since it yields, for instance, see \citet[Prop. 2.1.2]{bicklEtAl1998},
  \[
    \sum_{j=1}^n\log \frac{\d P_{\theta_{1/\sqrt n}}}{\d P_\theta}(X_j)=\frac{1}{\sqrt n}\sum_{j=1}^nA_\theta b(X_j)-\frac{1}{2}\|A_\theta b\|_{L^2(P_\theta)}+R_n
  \]
  for a remainder $R_n$ that converges in $P_\theta^{\otimes n}$-probability to zero. Hence, the LAIN property \eqref{eqLAN} is statisfied with rate $1/\sqrt n$ and with the score operator $A_\theta$ mapping into the Hilbert space $H=L_0^2(P_\theta)$. The convergence of the finite dimensional distributions in Definition~\ref{defReg} follows from the Cram\'er--Wold device and the linearity of $A_\theta$.
\end{proof}
Note that $L_0^2(P_\nu)$ is the orthogonal complement of $\lin 1$ and thus it is a closed subspace of the Hilbert space $L^2(P_\nu)$. The operator $A_\theta$ maps directions $b\in\dot\Theta_\theta$ into score functions at $P_\theta$ and thus it is called \textit{score operator} which explains the name given in the general case. It generates the tangent set $\dot{\mathcal P}_{P_\theta}=A_\theta\dot\Theta_\theta$ of the model $\mathcal P$ at $P_\theta$. Note that the range of $A_\theta$ is a subset of the maximal tangent set as the following example shows.
\begin{example}[Maximal tangent set]\label{exMaxTang}
  Let $\mathcal P$ be the model of all probability measures on some sample space. The maximal tangent set of the model $\mathcal P$ at some distribution $P$ is given by $L^2_0(P)$. This can be seen as follows: Score functions are necessarily centered and square integrable. For any score $g\in L^2_0(P)$ a one-dimensional submodel is $t\to c(t)k\big(tg(x)\big)\d P(x)$ with a $C^2(\R)$-function $k:\R\to\R_+$ which satisfies $k(0)=k'(0)=1$ and such that $k'/k$ is bounded and with normalization constant $c(t)=\|k(tb)\|_{L^1(\nu)}^{-1}$, for instance, $k(y)=2/(1+e^{-2y})$ \cite[cf. ][Ex. 25.16]{vanderVaart1998}.
\end{example}
Theorem~\ref{thmScore} yields then the following convolution theorem, which was already obtained by \citet{vanderVaart1991}. 
\begin{corollary}\label{corScore}
  Suppose the product model with marginal distributions $\mathcal P=\{P_\theta|\theta\in\Theta\}$ satisfies Assumption~\ref{assRegularity} and let $\chi:\Theta\to \R^d$ be pathwise differentiable with respect to $\dot\Theta_\theta$. The map $\psi:\mathcal P\to\R^d$ is differentiable at $P_\theta$ relative to the tangent set $\dot{\mathcal P}_{P_\theta}=A_\theta\dot\Theta_\theta$ if and only if each coordinate function of $\tilde\chi_\theta$ is contained in the range of the adjoint score operator $A_\theta^\star:L^2_0(P_\theta)\to\overline{\lin}\dot\Theta_\theta$. In this case the efficient influence function is given by $\tilde\psi_{\theta}=(A_\theta^\star)^\dagger \tilde\chi_\theta$.
\end{corollary}
In particular, for $\zeta\in\ran A_\theta^\star$ the asymptotic covariance matrix of every regular estimator is bounded from below by
\[
  \E_\theta\big[\tilde\psi_{P_\theta}\tilde\psi_{P_\theta}^\top\big]
  =\E_\theta\Big[\big((A_\theta^\star)^\dagger \tilde\chi_\theta\big)\big((A_\theta^\star)^\dagger \tilde\chi_\theta\big)^\top\Big].
\]
If $\tilde\chi_\theta\notin\ran A_\theta^\star$, \citet{vanderVaart1991} shows that there exists no regular estimator of the functional $\chi(\theta)$.

In the i.i.d. case we find the following statistical interpretation of Proposition~\ref{propApprox}, adopting the Cram\'er--Rao point of view. 
Let $\mathcal G$ be a dense subset of $L^2_0(P_\theta)$ and let $\chi(\theta)$ be a one-dimensional derived parameter with gradient $\tilde\chi_\theta$. Consider an approximating sequence $\chi_n\to\tilde\chi_\theta$ satisfying $\mathcal G\ni g_n:=(A_\theta^\star)^{\dagger}\chi_n\to\tilde\psi_{P_\theta}$. Assuming $\mathcal G\subset\ran A_\theta$, we can define $b_n:=A_\theta^{\dagger}g_n=I^{\dagger}\chi_n$ where $I:=A_\theta^\star A_\theta$ is the information operator. The information bound can be read as a Cram\'er--Rao bound in the least favorable submodel
\begin{align}
    \E[\tilde\psi_{P_\theta}^2]
    &=\sup_{g\in\lin\dot{\mathcal P}_{P_\theta}}\frac{\scapro{\tilde\psi_{P_\theta}}g_{P_\theta}^2}{\scapro g g_{P_\theta}}
    =\sup_{b\in\lin\dot\Theta_\theta}\frac{\scapro{\tilde\chi_\theta}{b}_\theta^2}{\scapro {A_\theta b}{A_\theta b}}_{P_\theta}\notag\\
    &\ge\frac{\big(\scapro{\chi_n}{b_n}_\theta-\scapro{\tilde\chi_\theta-\chi_n}{b_n}_\theta\big)^2}{\scapro{A_\theta b_n}{A_\theta b_n}_{P_\theta}},\label{eqCR}
\end{align}
where we plugged in the direction $b_n=I^{\dagger}\chi_n$. The term $\scapro{\chi_n}{b_n}_\theta^2/\scapro{A_\theta b_n}{A_\theta b_n}_{P_\theta}=\scapro{g_n}{g_n}_{P_\theta}$ is the Cram\'er--Rao bound for the estimation problem of a functional, which approximates $\chi(\theta)$, with gradient $\chi_n$. The approximation error $\scapro{\tilde\chi_\theta-\chi_n}{b_n}_\theta$ should be understood as bias. Since $b_n$ does not have to be bounded, $\chi_n\to\tilde\chi_\theta$ is not sufficient to conclude that the bias vanishes. However, Proposition~\ref{propApprox}(ii) implies that this error converges to zero owing to the Cauchy--Schwarz inequality:
\begin{align*}
  |\scapro{\tilde\chi_\theta-\chi_n}{b_n}_\theta|
  &=|\scapro{(A_\theta^\star)^{\dagger}(\tilde\chi_\theta-\chi_n)}
  {A_\theta b_n}_{P_\theta}|\\
  &\le\|\tilde \psi_{P_\theta}-g_n\|_{P_\theta}\|g_n\|_{P_\theta}\to0.
\end{align*}
Hence, the Cram\'er--Rao bound \eqref{eqCR} converges to the information bound $\scapro{\tilde\psi_{P_\theta}}{\tilde\psi_{P_\theta}}_{P_\theta}$. A similar perspective was taken by \citet[Lem. 3]{soehlTrabs2012}.

\section{Deconvolution}\label{secDecon}
Let us discuss the previous results in the classical nonparametric deconvolution setup, which has many applications, e.g., measurement-error problems (see \cite{meister2009}). We observe an i.i.d. sample 
\begin{equation}\label{eqDeconModel}
  Y_i=X_i+\eps_i,\quad i=1,\dots,n.
\end{equation}
Let $X_i$ and $\eps_i$ be independent and have distributions $\nu$ and $\mu$, respectively. If $\mu$ is known, the model is $\mathcal P=\{P_\nu=\nu\ast\mu|\nu\in\Theta\}$ where the parameter set $\Theta$ is given by the set of all probability measures. We aim for a convolution theorem for estimating the linear functional $\psi(P_\nu)=\chi(\nu):=\int\zeta\d\nu$ with $\zeta\in L^2(\nu)$. One of the most interesting examples is the estimation of the distribution function of $X_1$, corresponding to $\zeta=\ind_{(-\infty,t]}$ for $t\in\R$.

In a general linear indirect density estimation setting, a convolution theorem was already proved by \citet{vanRooijEtAll1999}, who use the spectral decomposition of the operator. Their approach applies however only for a restricted class of functionals, depending on the polar decomposition, and they need an abstract condition on the density of $\nu$ which is difficult to verify. It implicitly assumes an appropriate decay behavior on this density. Their application to the deconvolution setting is restricted to a specific example. Studying deconvolution in more detail, \citet{soehlTrabs2012} have shown an information bound, assuming a polynomial decay behavior of a sufficiently regular Lebesgue density of $\nu$ and a bit more than second moments. They described the class of admissible functionals analytically, including the estimation of the distribution function of $\nu$. Here, we are able to relax the conditions on $\nu$ and $\mu$ considerably, see Theorem~\ref{thmDecon} and Remark~\ref{exDecon3} below.

For any $\nu\in\Theta$ we choose the tangent space
\[
  \dot\Theta_\nu:=L_0^2(\nu)=\overline\lin\dot\Theta_\nu.
\]
According to Example~\ref{exMaxTang}, $\dot\Theta_\nu$ coincides with the maximal tangent set for direct observations. For any direction $b\in\dot\Theta_\nu$ and some sufficiently small $\tau>0$ the path $[0,\tau)\ni t\to\nu_t$ where $\frac{\d\nu_t}{\d\nu}=k(tb)/\int k(tb)\d\nu$ with $k:\R\to\R_+$ as in Example~\ref{exMaxTang} is a submodel of $\Theta$ with $b=\frac{\partial}{\partial t}|_{t=0}\log(\d\nu_t)$. Using $|k(tb)|\le t|b|\in L^2(\nu)$ and dominated convergence, the pathwise derivative of $\chi$ along $t\mapsto\nu_t$ at $t=0$ is given by
\begin{align*}
  \lim_{t\to0}t^{-1}(\chi(\nu_t)-\chi(\nu))
  &=\lim_{t\to0}\int\zeta(x)t^{-1}\Big(\frac{\d \nu_t}{\d\nu}(x)-1\Big)\d\nu(x)\\
  &=\int\zeta(x)b(x)\nu(\d x)
  =\scapro\zeta b_\nu=:\dot\chi_\nu b.
\end{align*}
Hence, the derivative can be represented by $\dot\chi_\nu b=\scapro{\tilde\chi_\nu}b_\nu$ for $\tilde\chi_\nu=\zeta-\int\zeta\d\nu\in\dot\Theta_\nu$. The path $t\mapsto\nu_t$ induces a regular submodel $t\mapsto P_{\nu_t}=\nu_t\ast \mu$ which is shown by the following lemma.
\begin{lemma}\label{lemRegDecon}
  For any nonzero $b\in\dot\Theta_\nu=L_0^2(\nu)$ the submodel $[0,\tau)\ni t\mapsto P_{\nu_t}=\nu_t\ast \mu$, for $\tau>0$ sufficiently small, is Hellinger differentiable, that is \eqref{eqHellingerDeriv} holds with continuous score operator 
  \begin{equation}\label{eqADecon}
    A_\nu:\quad \dot\Theta_\nu\to L^2_0(P_\nu),\quad b\mapsto\E\big[b(X)|X+\eps\big]=\frac{\d\big((b\nu)\ast\mu\big)}{\d P_{\nu}},
  \end{equation}
  where the expectation is taken with respect to the product measure measure $P^{(X,\eps)}=\nu\otimes\mu$.
\end{lemma}
\begin{proof}
  First, note that the (signed) measure $(f\nu)\ast\mu$ is absolutely continuous with respect to $\nu\ast\mu$ for any $f\in L^1(\nu)$, written as $(f\nu)\ast\mu\ll\nu\ast\mu$. In particular, the Radon--Nikodym density in \eqref{eqADecon} is well defined and $P_{\nu_t}\ll P_\nu$ for all $t>0$. Let us write $\E_t[\cdot]$ for the expectation under $P_t^{(X,\eps)}=\nu_t\otimes\mu$. We define $p_t(y):=\frac{\d P_{\nu_t}}{\d P_\nu}(y)$, $n_t(x):=\frac{\d \nu_t}{\d\nu}(x)=k(tb(x))/\int k(tb)\d\nu$. Let $\R\times\mathscr B(\R)\ni(y,A)\mapsto\kappa_{X,X+\eps}(y,A)$ be the regular conditional probability of $P^{(X,\eps)}(X\in\cdot|X+\eps)$ that is
  \[
    \kappa_{X,X+\eps}(y,A)=P^{(X,\eps)}(X\in A|X+\eps=y)
  \]
  for $P^\nu$-a.e. $y\in\R$ and all $A\in\mathscr B(\R)$. We claim
  \begin{equation}\label{eqPyt}
    p_t(Y)=\E_0\big[n_t(X)\big|X+\eps=Y\big]=\int n_t(x)\kappa_{X,X+\eps}(Y,\d x)\quad P_\nu-a.s.
  \end{equation}
  To verify \eqref{eqPyt}, we note for any Borel set $A\in\mathscr B(\R)$
  \begin{align}
    \E_0[\ind_A(Y)p_t(Y)]
    =\E_t[\ind_A(Y)]
    &=\E_0[\ind_A(X+\eps)n_t(X)]\notag\\
    &=\E_0\big[\ind_A(Y)\E_0[n_t(X)|X+\eps=Y]\big]\label{eqCondExp}
  \end{align}
  which shows the first equality in \eqref{eqPyt}. The second one follows from the choice of $\kappa_{X,X+\eps}$.

  We will show regularity of the submodel $(-\tau,\tau)\ni t\mapsto P_{\nu_t}=\nu_t\ast \mu$ for a sufficiently small $\tau>0$ by applying Proposition 2.1.1 in \cite{bicklEtAl1998}. Using the properties of $k$,
  \[
    \dot n_t(x):=\frac{\partial}{\partial t}n_t(x)=\frac{b(x)k'(tb(x))\int k(tb)\d\nu-k(tb(x))\int bk'(tb)\d\nu}{(\int k(tb)\d\nu)^2}
  \]
  can be uniformly in $t\in(-\tau,\tau)$ bounded by $c_b(b(x)+1)$ for a constant $c_b>0$, depending on $b$, and it is continuous in $t$ on $(-\tau,\tau)$ for some sufficiently small $\tau>0$. Since $b\in L^2(\nu)\subset L^2(\nu\otimes\mu)$, dominated convergence and \eqref{eqPyt} yield that $p_t(y)$ is continuously differentiable in $t\in(-\tau,\tau)$ for $P_\nu$-almost all $y\in\R$ with derivative
  \[
    \dot p_t(y)=\frac{\partial}{\partial t}p_t(y)=\int \dot n_t(x)\kappa_{X,X+\eps}(y,\d x)=\E_0[\dot n_t(X)|X+\eps=y].
  \]
  By Jensen's inequality we see that
  \begin{align}
    \|\dot p_t\|_{L^2(P_t)}^2
    &=\E_0\big[p_t(Y)\big|\E_0[\dot n_t(X)|X+\eps=Y]\big|^2\big]\label{eqFishInfoDecon}\\
    &\le\E_0\big[p_t(Y)|\dot n_t(X)|^2\big].\notag
  \end{align}
  Since $n_t(x)$ can be bounded uniformly in $t\in(-\tau,\tau)$ and $x\in\R$, the density $p_t(Y)$ is $P_\nu$-a.s. bounded by some constant $C>0$ owing to \eqref{eqPyt}. Therefore, we conclude from the previous estimate together with the bound $|\dot n_t|\le c_b(b+1)$ and $b\in L^2(\nu)$ that
  \[
    \|\dot p_t\|_{L^2(P_t)}^2\le C\E_0[|\dot n_t(X)|^2]\le Cc_b^2\E_0[(b(X)+1)^2]<\infty.
  \]
  In particular, the Fisher information $I_t:=\E_t[\dot p_t(Y)^2]$ is finite. Using \eqref{eqFishInfoDecon}, we infer that $I_t$ is continuous. Noting that 
  \begin{align*}
    \frac{\partial}{\partial t}\Big|_{t=0}\int k(tb)\d\nu=\int b\d\nu=0\quad\text{and thus}\quad
    \dot n_0(y)=\frac{\partial}{\partial t}\Big|_{t=0}n_t(x)=b(x),
  \end{align*}
  we have $I_0=\E_0[b(X)^2]$ and $I_t$ is therefore nonzero for $b\neq 0$ and $t$ small enough. In combination with the continuous differentiability of $p_t$, Proposition 2.1.1 in \cite{bicklEtAl1998} yields the Hellinger differentiability~\eqref{eqHellingerDeriv} at $t=0$ with derivative $\frac{1}{2}\dot p_0$. We obtain the score operator
  \[
    A_\nu b:=\dot p_0=\E_0[b(X)|X+\eps].
  \]
  To see that $A_\nu:\dot\Theta_\nu\to L_0^2(P_\nu)$ is well defined and continuous, we again use Jensen's inequality which yields
  \[
    \E_0[A_\nu b]=\E_0[b(X)]=0\quad\text{and}\quad \E_0[|A_\nu b|^2]\le\E_0[|b(X)|^2]=\|b\|_{L^2(\nu)}^2.
  \]
  Finally, a similar calculation as \eqref{eqCondExp} shows $A_\nu b=\E_0[b(X)|X+\eps]=\frac{\d((b\nu)\ast\mu)}{\d P_{\nu}}$.
\end{proof}

Lemma~\ref{lemRegDecon} shows that Assumption~\ref{assRegularity} is satisfied and thus Lemma~\ref{lem:hellinger} yields regularity of the deconvolution model with rate $r_n=1/\sqrt n$. In order to apply Corollary~\ref{corScore}, we have to determine the adjoint of the score operator. For any $g\in L_0^2(P_\nu)$ and any $b\in \dot\Theta_\theta\subset L^2(\nu)$ the map $\R^2\ni (x,y)\mapsto g(x+y)b(x)$ is $\nu\otimes\mu$-integrable due to the Cauchy--Schwarz inequality. $(b\nu)\ast\mu\ll P_\nu$ and Fubini's theorem thus yield
\begin{align*}
  \scapro{A_\nu b}g_{P_\nu}
  &=\int(A_\nu b)g\d P_\nu
  =\int g\d\big((b\nu)\ast\mu\big)\\
  &=\int\int g(x+y)b(x)\nu(\d x)\mu(\d y)
  =\scapro{\mu(-\cdot)\ast g}b_\nu.
\end{align*}
Noting that Jensen's inequality shows
\[
  \|\mu(-\cdot)\ast g\|_{L^2(\nu)}^2
  =\int\Big(\int g(x+y)\mu(\d y)\Big)^2\nu(\d x)
  \le\int g^2\d(\nu\ast\mu)
  =\|g\|_{L^2(P_\nu)}^2
\]
and that $\int (\mu(-\cdot)\ast g)\d\nu=\int g\d(\nu\ast\mu)=0$, the adjoint score operator equals
\begin{equation}\label{eqAstarDecon}
  A_\nu^\star:\quad L^2_0(P_\nu)\to\dot\Theta_\nu,\quad
  g\mapsto \mu(-\cdot)\ast g.
\end{equation}
Under weak conditions on the measures $\nu$ and $\mu$ we conclude the following convolution theorem. Thereby we extend the definition of the Fourier transform to finite measure $\mu$ on the Borel measurable space $(\R,\mathscr B(\R))$ by $\F\mu(u)=\int e^{iux}\mu(\d x)$.
\begin{theorem}\label{thmDecon}
  In the deconvolution model~\eqref{eqDeconModel} suppose that $\phi_\eps(u):=\E[e^{iu\eps_1}]=\F\mu(u)\neq0$ for all $u\in\R$ and that $\nu$ admits a Lebesgue density. Then $A_\nu^\star$ as given in \eqref{eqAstarDecon} is injective. Let moreover $\zeta^{(1)},\dots,\zeta^{(d)}\in L^2(\nu)$ satisfy $\zeta^{(j)}\in\ran A_\nu^\star$ for $j=1,\dots,d$ and $d\in\N$. Then the limit distribution of any estimator of the parameter $(\int\zeta^{(1)}\d\nu,\dots,\int\zeta^{(d)}\d\nu)$ which is regular with respect to the rate $n^{-1/2}$ equals $\mathcal N(0,\Sigma)\ast M$ for some Borel probability measure $M$ and with covariance matrix $\Sigma\in\R^{d\times d}$ given by
  \begin{equation}\label{eqInfoBoundDeconv}
    \Sigma_{j,k}=\int\Big((A_\nu^\star)^{-1}\zeta^{(j)}\Big)\Big((A_\nu^\star)^{-1}\zeta^{(k)}\Big)\d P_\nu-\Big(\int\zeta^{(j)}\d\nu\Big)\Big(\int\zeta^{(k)}\d\nu\Big)
  \end{equation}  
  for $j,k=1,\dots,d$.
\end{theorem}
\begin{proof}
  Let us first show that on the assumptions the adjoint operator $A_\nu^\star$ is injective. Since $\nu$ admits a Lebesgue density, the equivalence classes with respect to the Lebesgue measure embed into the equivalence classes with respect to $\nu$ and with respect to $P_\nu$. Hence, we can consider the subset $\mathcal G:=L^2(\R)\cap L^2_0(P_\nu)$, which is dense in $L^2_0(P_\nu)$ (cf. Section~\ref{secProofAdjoint} (ii)). Since the kernel of the continuous operator $A_\nu^\star$ is closed, it is sufficient to show that the restricted operator $A_\nu^\star|_{\mathcal G}$ is injective. For any $g\in\mathcal G$ it holds $0=A_\nu^\star g=\mu(-\cdot)\ast g$ if and only if $0=\F[\mu(-\cdot)\ast g]=\phi_\eps(-\cdot)\F g$ which is equivalent to $\F g=0$ since $|\phi_\eps|>0$ by assumption. Hence, the kernel of $A_\nu^\star$ equals $\{0\}$.

  To infer the information bound, recall that the gradient of the linear functional $\int\zeta^{(j)}\d \nu$ is $\tilde\chi_\nu^{(j)}=\zeta^{(j)}-\int\zeta^{(j)}\d\nu$. Furthermore, note that $\zeta^{(j)}\in\ran A_\nu^\star$ implies $\tilde\chi_\nu^{(j)}\in\ran A_\nu^\star$ since $A_\nu^\star a=\mu(-\cdot)\ast a=\int a\d\mu=a$ for any real number $a\in\R$. Therefore, Lemma~\ref{lemRegDecon}, Corollary~\ref{corScore} and the injectivity of $A_\nu^\star$ yield the vector of efficient influence functions
  \[
    \tilde\psi_{P_\nu}^{(j)}=(A_\nu^\star)^{-1}\tilde\chi_\nu^{(j)}=(A_\nu^\star)^{-1}\zeta^{(j)}-\int\zeta^{(j)}\d\nu,\quad j=1,\dots,d,
  \]
  and the assertion follows from Theorem~\ref{thmScore}. 
\end{proof}

\begin{remark}
  The assumption $\phi_\eps(u)\neq0, u\in\R,$ is not sufficient for the injectivity of $A_\nu^\star$ as the following counterexample shows: Let $\nu=\delta_0$ be the Dirac measure in zero and  $\mu=\mathcal N(0,1)$ be standard normal such that $P_\nu=\nu\ast\mu=\mathcal N(0,1)$. Consider $g\in L^2_0(P_\nu)$ with $g(x)=x^3, x\in\R$. Then, $A_\nu^\star g(x)=\E[(x+\eps_1)^3]$ is zero at the origin. Hence, $A_\nu^\star g=0$  $\nu$-a.s. and $0\neq g\in\ker A_\nu^\star$. 
  A sufficient condition for $A_\nu^\star$ being injective is given in Theorem~\ref{thmDecon} by assuming additionally a Lebesgue density of $\nu$, which is a natural assumption. In particular, we obtain $\overline\ran A_\nu=L^2_0(P_\nu)$ implying that the tangent space $\dot{\mathcal P}_{P_\nu}=A_\nu\dot\Theta_\nu$ is dense in the set of all score functions. Without injectivity the convolution theorem remains true if the inverse $(A_\nu^\star)^{-1}$ in ~\eqref{eqInfoBoundDeconv} is replaced by the Moore--Penrose pseudoinverse $(A_\nu^\star)^\dagger$. 
\end{remark}

In view of Theorem~\ref{thmScore} our condition $\zeta^{(j)}\in\ran A_\nu^\star, j=1,\dots,d,$ is necessary and sufficient for the regularity of the parameter. The remaining question is under which conditions $\ind_{(-\infty,t]}\in\ran A_\nu^\star, t\in\R,$ and how the pre-image looks like in this case. Applying the approach by \citet{nicklReiss2012}, we will give an answer in the more involved L\'evy setting. Under certain assumptions on $\mu$ it can be carried over to the deconvolution case and leads to a similar result, but with weaker conditions on $\mu$ than by \citet{soehlTrabs2012} (see Remark~\ref{exDecon3} below). In particular, these conditions will imply that $\ind_{(-\infty,t]}$ can be estimated with parametric rate.  

\section{Application to L\'evy processes}\label{secLevy}

\subsection{Setting and regularity}
Recall that a L\'evy process $(L_t)_{t\ge0}$ is a stochastic process which is stochastically continuous with $L_0=0$ and which has stationary and independent increments. Let $(L_t)$ be real-valued. For some distance $\Delta>0$ we observe the L\'evy process at equidistant time points $t_k=\Delta k$ with $k=0,\dots,n$. In the so called low-frequency regime $\Delta$ remains fixed as $n$ goes to infinity. The L\'evy process is uniquely determined by its characteristic triplet consisting of the volatility $\sigma^2\ge0$, of the drift parameter $\gamma\in\R$ and of the L\'evy or jump measure $\nu$ on $(\R,\mathscr B(\R))$ which satisfies $\int_{\R}(|x|^2\wedge1)\nu(\d x)<\infty$ and $\nu(\{0\})=0$ \citep[cf.][Chap. 2]{sato1999}. Our aim is to derive a convolution theorem for the estimation of the linear functional of the jump measure
\begin{equation}\label{eqFunctional}
  \chi(\nu):=\int\zeta\d\nu\quad\text{for }\zeta\in L^1(\nu)\cap L^2(\nu).
\end{equation} 
If $\zeta$ is $\R^d$-valued, the scalar product in \eqref{eqFunctional} has to be interpreted coordinatewise. As a relevant example the reader should have in mind the generalized distribution function of $\nu$. It corresponds to $\zeta=\ind_{(-\infty,t]}$ for $t<0$ and $\zeta=\ind_{[t,\infty)}$ for $t>0$. This adaptation of the standard distribution function is necessary owing to the possibly existing singularity of $\nu$ at the origin. In order for the estimation of $\chi(\nu)$ to be possible with parametric rate, we restrict on processes with finite variation in view of the lower bounds by \citet{neumannReiss2009}. That means $\sigma^2=0$ and $\int_{\R}(|x|\wedge1)\nu(\d x)<\infty$ are assumed. For a recent review on the statistical inference on L\'evy processes we refer to \citet{reiss2013}.

Due to the stationary and independent increments of $(L_t)$, the random variables $Y_k:=L_{\Delta k}-L_{\Delta(k-1)}, k=1,\dots,n,$ are independent and identically distributed. Their characteristic function is given by the L\'evy-Khintchine formula
\begin{equation}\label{eqLevyKhintchine}
  \phi_\nu(u)=\E\big[e^{iuL_\Delta}\big]
  =\exp\Big(\Delta\big(i\gamma u+\int(e^{iux}-1)\nu(\d x)\big)\Big).
\end{equation}
Fixing the drift $\gamma$, the model is given by
\begin{align}\label{eqModel}
  \mathcal P&=\big\{P_\nu=\F^{-1}\phi_\nu|\nu\in\Theta\big\}\quad\text{with}\\
  \Theta&=\Big\{\nu \text{ jump measure on }(\R,\mathscr B(\R))\Big|\int(|x|\wedge1)\nu(\d x)<\infty\Big\}.\notag
\end{align}

Compared to tangents at the set of probability measures in Example~\ref{exMaxTang}, directions for the L\'evy measures do not need to be centered since L\'evy measures are not normalized. In general, jump measures are even not finite such that $L^2(\nu)$, which gives the Hilbert space structure, is still too large. We should intersect with $L^1(\nu)$ to include linear functionals as \eqref{eqFunctional}. Hence, we define the tangent space at $\nu\in\Theta$ as 
\begin{equation}\label{eqTangentSpace}
  \dot\Theta_\nu:=L^1(\nu)\cap L^2(\nu)
  =\lin\dot\Theta_\nu.
\end{equation}
Using the function $k(y)=2/(1+e^{-2y})$ from Example~\ref{exMaxTang}, for any $b\in\dot\Theta_\nu$ the path $[0,1)\ni t\mapsto\nu_t$ with $\frac{\d\nu_t}{\d\nu}(x)=k\big(tb(x)\big)$ is contained in $\Theta$ and satisfies
\[
  b(x)=\frac{\partial}{\partial t}\Big|_{t=0}\log\Big(\frac{\d\nu_t}{\d\nu}(x)\Big).
\]
On this path the derivative of the functional~\eqref{eqFunctional} can be calculated with use of dominated convergence, noting that $|k(tb)-1|\le t|b|\in L^2(\nu)$. Hence,
\begin{align*}
  \lim_{t\to0}t^{-1}(\chi(\nu_t)-\chi(\nu))
  &=\lim_{t\to0}\int\zeta(x)t^{-1}\Big(\frac{\d \nu_t}{\d\nu}(x)-1\Big)\d\nu(x)\\
  &=\int\zeta(x)b(x)\nu(\d x)
  =\scapro\zeta b_\nu=:\dot\chi_\nu b
\end{align*}
and thus the gradient is given by $\tilde\chi_\nu=\zeta$. Compared to the deconvolution setting, we do not need to center $\tilde\chi_\nu$ because the total mass of the L\'evy measure is allowed to change along the path. 

To apply Corollary~\ref{corScore}, we need to verify Assumption~\ref{assRegularity} for the L\'evy model. By the L\'evy--Khintchine representation the laws $P_{\nu_t}$ satisfy
\begin{align}
  P_{\nu_t}&=\F^{-1}\Big[\exp\Big(\Delta\Big(i\gamma u+\int(e^{iux}-1)\nu_t(\d x)\Big)\Big)\Big]\notag\\
  &=\F^{-1}\Big[\exp\Big(\Delta\int(e^{iux}-1)\Big(k\big(tb(x)\big)-1\Big)\nu(\d x)\Big)\phi_\nu(u)\Big].\label{eqDens}
\end{align}
Owing to $(k(tb)-1)\in L^1(\nu)$, the measure $P_{\nu_t}$ is a convolution of $P_\nu$ and a compound Poisson type measure with signed jump measure $\Delta (k(tb)-1)\d\nu$. To see that the submodel $t\mapsto P_{\nu_t}$ is dominated, we check that the Hellinger distance of the jump measures $\int(\sqrt{\d\nu_t}-\sqrt{\d\nu})^2=\int(\sqrt{k(tb(x))}-1)^2\nu(\d x)$ is finite for all $t$. Since the drift $\gamma$ remains constant, Theorem 33.1 in \cite{sato1999} yields absolute continuity of $P_{\nu_t}$ with respect to $P_\nu$, denoted as $P_{\nu_t}\ll P_\nu$, for all $t$. To find the Hellinger derivative of the path $t\mapsto \frac{\d P_{\nu_t}}{\d P_\nu}$ at $t=0$, we note by dominated convergence
\begin{align*}
  &\F^{-1}\Big[\frac{\partial}{\partial t}\Big|_{t=0}\exp\Big(\Delta\int(e^{iux}-1)\big(k\big(tb(x)\big)-1\big)\nu(\d x)\Big)\phi_\nu(u)\Big]\\
  =&\Delta\F^{-1}\Big[\phi_\nu(u)\Big(\int(e^{iux}-1)b(x)\nu(\d x)\Big)\Big]\\
  =&\Delta\Big(P_\nu\ast(b\nu)-\int b(x)\nu(\d x)P_\nu\Big).
\end{align*}
This indicates how the score operator should look like. The following proposition determines the score operator $A_\nu$ and shows Hellinger regularity of the parametric submodel $t\mapsto P_{\nu_t}$. This is the key result to apply the theory of Section~\ref{secLAN} to the L\'evy model.
\begin{proposition}\label{propRegularity}
  Let the model $\mathcal P$ be given by \eqref{eqModel} with the tangent space $\dot\Theta_\nu$ at $\nu\in\Theta$ as defined in \eqref{eqTangentSpace}. Then $P_\nu\ast(b\nu)\ll P_\nu$ for all $b\in \dot\Theta_\nu\cap L^\infty(\nu)$. Moreover, the linear operator 
  \begin{equation}\label{eqScoreOp}
    A_\nu \big|_{\dot\Theta_\nu\cap L^\infty(\nu)}:\dot\Theta_\nu\cap L^\infty(\nu)\to L^2_0(P_\nu),\quad b\mapsto\Delta\frac{\d \big(P_\nu\ast(b\nu)\big)-\big(\int b\d\nu\big)\d P_\nu}{\d P_\nu}
  \end{equation}
  is bounded. $\dot\Theta_\nu\cap L^\infty(\nu)$ is dense in $\dot\Theta_\theta$ and thus $A_\nu: \dot\Theta_\nu\to L^2_0(P_\nu)$ can be defined as its unique continuous extension. Then for all $b\in\dot\Theta_\nu$ the associated submodel $[0,1)\ni t\mapsto P_{\nu_t}$ is Hellinger differentiable at zero with derivative $A_\nu b$, that means \eqref{eqHellingerDeriv} is fulfilled.
\end{proposition}
The proof of this proposition is given in Section~\ref{secProofReg}. An essential ingredient is an estimate of the Hellinger integral of two infinitely divisible distributions by \citet{Liese1987}. More precisely, his results imply (for details see Section~\ref{secProofReg})
\begin{align}
  \int\Big(\frac{\d P_{\nu_t}}{\d P_\nu}\Big)^2\d P_\nu
  \le\exp\Big(\frac{1}{2}\int\big(\frac{\d\nu_t}{\d\nu}-1\big)^2\d\nu\Big)
  \le\exp\Big(\frac{1}{2}t^2\|b\|_{L^2(\nu)}^2\Big).\label{eqH2dist}
\end{align}
\begin{remark}
  Relying on the semimartingale structure of the model, an alternative strategy to prove Proposition~\ref{propRegularity} is as follows: Observing $L_\Delta$ is a sub-experiment of observing $(L_t)_{0\le t\le \Delta}$ in continuous time. Hence, by proving Hellinger differentiability of the latter model, we find a score process, say $(V_t)_{0\le t\le\Delta}$, and the score operator is then given by $A_\nu b=\E[V_\Delta|L_\Delta]$.

  Theorem 2.34 by \citet{jacod1990} yield local differentiability of the experiment with continuous observations corresponding to the score process
  \[  V_t=\sum_{s\le t}b(\Delta L_s)-t\int b\d \nu.\]
  Noting that the Hellinger processes are deterministic \citep[cf.][Rem. IV.1.25]{jacodShiryaev2003}, local differentiability implies Hellinger differentiability, see \cite{jacod1989}.\footnote{Thanks to an anonymous referee for pointing out this approach.}
\end{remark}

Proposition~\ref{propRegularity} shows that the L\'evy model $\mathcal P$, defined in \eqref{eqModel} equipped with the tangent space $\dot\Theta_\nu$, given in \eqref{eqTangentSpace} satisfies Assumption~\ref{assRegularity}. In particular, it is a regular indirect model at any $\nu\in\Theta$ with respect to the rate $n^{-1/2}$ by Lemma~\ref{lem:hellinger}. Having in mind the regularity Lemma~\ref{lemRegDecon} in the deconvolution model, the score operators look very similar in both models. Since the gradient does not have to integrate to zero in the L\'evy model, the centering is incorporated in the operator $A_\nu$. Apart from that the convolution structure is the same. Therefore, the L\'evy model can locally be weakly approximated with a linear white noise model whose operator is of convolution type.


By Proposition~\ref{propRegularity} the score operator $A_\nu$ is characterized by \eqref{eqScoreOp}. To prove information bounds, we will combine this result with Proposition~\ref{propApprox} which shows that it is sufficient to study $A_\nu^\star$ on a nicely chosen, dense subset of $L_0^2(P_\nu)$. Then Theorem~\ref{thmScore} provides the convolution theorem for all $\zeta\in\ran A_\theta^\star$. In the following we will discuss L\'evy processes with finite and infinite jump activity separately because the analytical properties of the score operator are quite different: In the compound Poisson case the inverse adjoint score operator can be explicitly expressed as a convolution with a finite signed measure. If the jump intensity is infinite, the distribution $P_\nu$ possesses a Lebesgue density and thus $A_\nu^\star$ will be a smoothing operator.

\subsection{Compound Poisson processes}
Let $(L_t)$ be a compound Poisson process with jump intensity $\lambda:=\nu(\R)<\infty$. Consequently, the tangent space simplifies to $\dot\Theta_\theta=L^2(\nu)$ and the measure $P_\nu$ can be written as the convolution exponential \citep[cf.][Rem. 27.3]{sato1999}
\begin{equation}\label{eqConExp}
  P_\nu=\delta_{\Delta\gamma}\ast\Big(e^{-\Delta\lambda}\sum_{k=0}^\infty\frac{\Delta^k}{k!}\nu^{\ast k}\Big),
\end{equation}
where $\delta_x$ denotes the Dirac measure in $x\in\R$. Define the subsets 
\begin{align}\label{eqSubsetsCP}
  \mathcal G&:=L^\infty(P_\nu)\cap L^2_0(P_\nu)\subset L^2_0(P_\nu)\quad\text{ and }\\
  \mathcal H&:=\big\{h\in L^2(\nu)\big|\sup_{k=0,1,\dots}\|h\|_{L^\infty(\nu^{\ast k})}<\infty\big\}\subset \dot\Theta_\nu  \notag
\end{align}
which are dense in $L^2_0(P_\nu)$ and $L^2(\nu)$, respectively. Let $g\in\mathcal G$ and $b\in\mathcal H$. By Proposition~\ref{propRegularity} we know $(b\nu)\ast P_\nu\ll P_\nu$ which implies $g\in L^\infty((b\nu)\ast P_\nu)$. Hence, $\int g\d P_\nu=0$ and Fubini's theorem yield
\begin{align}
  \scapro{A_\nu b}g_{P_\nu}
  =\int(A_\nu b)g\d P_\nu
  &=\Delta\int g\d\big(P_\nu\ast(b\nu)\big)-\Delta\Big(\int b\d\nu\Big)\Big(\int g\d P_\nu\Big)\notag\\
  &=\Delta\int\int g(x+y)b(x)\d\nu(x)\d P_\nu(y)\notag\\
  &=\Delta\scapro{P_\nu(-\cdot)\ast g}b_\nu.\label{eqFubini}
\end{align}
Therefore, the adjoint score operator on $\mathcal G$ is $A_\nu^\star|_{\mathcal G}:\mathcal G\to L^2(\nu), g\mapsto\Delta P_\nu(-\cdot)\ast g$. 
\begin{lemma}\label{lemCPadjoint}
  The map $A_\nu^\star:\mathcal G\to \mathcal H,g\mapsto P_\nu(-\cdot)\ast g$ is well defined.
\end{lemma}
This lemma is proved in Section~\ref{secProofCPajoint}. Although the centering of $g\in L^2_0(P_\nu)$ implies $A_\nu^\star g(0)=\int g\d P_\nu=0$, it does not cause an additional constraint owing to $\nu(\{0\})=0$. In general, $A_\nu^\star$ is not injective as the following example shows:
\begin{example}[Poisson process]
  Setting $\nu=\delta_1, \gamma=0$ and $\Delta=\lambda=1$, the law $P_\nu=e^{-1}\sum_{k=0}^\infty\delta_k/(k!)$ is the Poisson distribution and the adjoint score operator is given by 
  \[A_\nu^\star g(x)=e^{-1}\sum_{k=0}^\infty g(x+k)/(k!),\quad x\in\R.\]
  Consider the function $g=\ind_{\{0\}}-2\ind_{\{1\}}+2\ind_{\{2\}}$ which is a nonzero element of $\mathcal G$ by construction. However, $A_\nu^\star g(1)=0$ and thus $0\neq g\in\ker A_\nu^\star$ contradicting injectivity.
\end{example}
As in the deconvolution model we assume therefore that $\nu$ admits a Lebesgue density concluding injectivity of $A_\nu^\star$ exactly as in Theorem~\ref{thmDecon}. Since $|\phi_\nu(u)|=e^{\Delta \int(\cos(ux)-1)\nu(\d x)}\ge e^{-2\Delta\lambda}$ for all $u\in\R$ by \eqref{eqLevyKhintchine}, the inverse of $A_\nu^\star$ is then the deconvolution operator $h\mapsto \Delta^{-1}\F^{-1}[1/\phi_\nu(-\cdot)]\ast h$ with the finite signed measure
\[
  \F^{-1}[1/\phi_\nu(-\cdot)]= \delta_{-\Delta\gamma}\ast\Big(e^{\Delta\lambda}\sum_{k=0}^\infty\frac{(-\Delta)^k}{k!}\nu(-\cdot)^{\ast k}\Big)
\]
which is well defined on $\mathcal H$. In particular, the pre-image of the indicator function $\zeta=\ind_{(-\infty,t]}$ (or equivalently $\ind_{(-\infty,t]}\ind_{\R\setminus\{0\}}$) is well defined for any $t\in\R$. Consequently, Corollary~\ref{corScore} and Theorem~\ref{thmScore} yield 
\begin{corollary}\label{corInfoBoundCP}
  Let $(L_t)$ be a pure jump process of compound Poisson type with jump measure $\nu$ which is absolutely continuous with respect to the Lebesgue measure and with drift $\gamma\in\R$. Then the limit distribution of any regular estimator of the distribution function $\R^d\ni(t_1,\dots,t_d)\mapsto(\nu((-\infty,t_1]),\dots,\nu((-\infty,t_d]))$, for $d\in\N,$ is a convolution $\mathcal N(0,\Sigma)\ast M$ for some Borel probability measure $M$ and with covariance matrix $\Sigma\in\R^{d\times d}$ given by
  \begin{equation*}
    \Sigma_{i,j}=\Delta^{-2}\int\Big(\F^{-1}\Big[\frac{1}{\phi_\nu(-\cdot)}\Big]\ast\ind_{(-\infty,t_i]}\Big)\Big(\F^{-1}\Big[\frac{1}{\phi_\nu(-\cdot)}\Big]\ast\ind_{(-\infty,t_j]}\Big)\d P_\nu
  \end{equation*}
  for $i,j\in\{1,\dots,d\}$. 
\end{corollary}
  
Considering the negative half line, this lower bound coincides with the asymptotic variance of the kernel estimator by \citet{nicklReiss2012}. An interesting deviation is obtained by restricting the model on compound Poisson processes with known jump intensity $\lambda>0$ as studied in the \textit{decompounding} problem by \citet{buchmannGrubel2003}. Similarly to the deconvolution model in Section~\ref{secDecon} the tangent space is then given by $L^2_0(\nu)$ and thus the gradient of the functional $\chi(\nu)=\nu((-\infty,t])$ equals $\tilde\chi_\nu=\ind_{(-\infty,t]}-\nu((-\infty,t])$. We obtain the smaller information bound, setting $d=1$ for simplicity,
\[
  \Delta^{-2}\int\Big(\F^{-1}\Big[\frac{1}{\phi_\nu(-\cdot)}\Big]\ast\ind_{(-\infty,t]}\Big)^2\d P_\nu-\Delta^{-2}\nu((-\infty,t])^2.
\]
That means an efficient estimator which ``knows'' the jump intensity should have a smaller variance than for unknown $\lambda$ and the statistical problem is significantly simpler. Indeed, the estimator from \cite{buchmannGrubel2003} is asymptotically normal with the above variance.

\subsection{L\'evy processes with infinite jump activity}
  
If the L\'evy process has infinite jump activity, the analysis is more difficult. However, we can profit from the absolute continuity of the infinite divisible distribution $P_\nu$ with respect to the Lebesgue measure \citep[Thm. 27.4]{sato1999}. To apply Fourier methods, we will again assume that $\nu$ admits a Lebesgue density which implies in particular that the set of Lebesgue-a.e. equivalence classes embeds into the $\nu$-a.e. and into the $P_\nu$-a.e. equivalence classes. Keeping the Hilbert space structure, we can then define 
\begin{align}
  \mathcal G:= H^\infty(\R)\cap L_0^2(P_\nu) \quad\text{and}\quad
  \mathcal H:= \big\{b\in H^\infty(\R)|b(0)=0\big\}\cap \dot\Theta_\nu,\label{eqSubsetsIA}
\end{align}
where $H^\infty(\R):=\bigcap_{s\ge0}H^s(\R)$ with Sobolev spaces $H^s(\R):=\{f\in L^2(\R)|\|(1+|u|^2)^{s/2}\F f(u)\|_{L^2}<\infty \}$ of regularity $s\ge0$. For $b\in\mathcal H$ the condition $b(0)=0$ should hold for the continuous version of $b$. To allow that the generalized distribution function of $\nu$ can be estimated with parametric rate, we concentrate on mildly ill-posed problems leading to the assumption that $|\phi_\nu(u)|$ decays polynomially as $|u|\to\infty$ \citep[cf.][]{neumannReiss2009}.
\begin{lemma}\label{lemAdjoint}
  Let the finite variation L\'evy process $(L_t)$ with $\nu\in\Theta$ have infinite jump activity satisfying $|\phi_\nu(u)|\gtrsim(1+|u|)^{-\beta}$ for some $\beta>0$ and let $\nu$ be absolutely continuous with respect to the Lebesgue measure. Then
  \begin{enumerate}
    \item on $\mathcal G$ from \eqref{eqSubsetsIA} the adjoint score operator $A_\nu^\star|_{\mathcal G}$ is a bijection onto $\mathcal H$ satisfying
    \begin{align}
      A_\nu^\star\big|_{\mathcal G}:\quad \mathcal G&\to\mathcal H\subset\dot\Theta_\nu,\qquad\quad g\mapsto \Delta(P_\nu(-\cdot)\ast g),\label{eqAdjointOp}\\
      (A_\nu^\star)^{-1}\big|_{\mathcal H}:\quad\mathcal H&\to\mathcal G\subset L^2_0(P_\nu),\quad \;\,b\mapsto\Delta^{-1}\F^{-1}\big[\F b/\phi_\nu(-\cdot)\big],\label{eqAdjointOpInv}
    \end{align}
    \item $\mathcal G$ is dense in $L^2_0(P_\nu)$ and $A_\nu^\star$ is the unique continuous extension of $A_\nu^\star\big|_{\mathcal G}$.
  \end{enumerate}
\end{lemma}
The proof of this lemma uses that by the polynomial decay $(1+|u|)^{-\beta}\lesssim|\phi_\nu(u)|\lesssim1$ both operators, $A_\nu^\star$ and $(A_\nu^\star)^{-1}$, are Fourier multipliers on Sobolev spaces and thus $H^\infty(\R)$-functions are mapped into $H^\infty(\R)$ again. The smoothness will be used to show $|(A_\nu^\star g)(x)|\lesssim1\wedge|x|$ for $x\in\R$ in order to verify that $A_\nu^\star$ is well defined. For details we refer to Section~\ref{secProofAdjoint}.

Comparing the adjoint score operator \eqref{eqAdjointOpInv} in the L\'evy model to $A_\nu^\star$ in the deconvolution model \eqref{eqADecon}, we see that both operators have exactly the same structure. To invert $A_\nu^\star$, we have to deconvolve with the observation measure itself in the L\'evy case. From our lower bounds perspective we clearly recover this \textit{auto-deconvolution} phenomenon, which was already described by \citet{belomestnyReiss2006} as well as \citet{nicklReiss2012}. For convenience we will write throughout $\F^{-1}[1/\phi_\nu(-\cdot)]\ast b=\F^{-1}[\F b/\phi_\nu(-\cdot)]$ which is justified in distributional sense. In combination with the results for the compound Poisson case Lemma~\ref{lemAdjoint} has two immediate consequences. 
\begin{remark}
  If the L\'evy process is of finite variation, has an absolutely continuous jump measure and has either finite jump activity or has a polynomial decreasing characteristic function, then
  \begin{enumerate}
    \item $A_\nu^\star$ is injective and therefore $\overline\ran A_\nu=L_0^2(P_\nu)$. This means that the tangent set $\dot{\mathcal P}_{P_\nu}=A_\nu\dot\Theta_\nu$ is dense in $L^2(P_\nu)$.
    \item For any linear functional $\chi(\nu)=\int\zeta\d\nu$ satisfying $\zeta\in\mathcal H$, where $\mathcal H\subset\dot\Theta_\nu$ is defined in \eqref{eqSubsetsCP} and \eqref{eqSubsetsIA}, respectively, the information bound is given by
    \[
      \Delta^{-2}\int\Big(\F^{-1}\Big[\frac{1}{\phi_\nu(-\cdot)}\Big]\ast\zeta\Big)^2\d P_\nu.
    \]
  \end{enumerate} 
\end{remark}
The subset $\mathcal H$ of arbitrary large Sobolev smoothness is obviously very restrictive. Let us extend the information bound to a larger class of functionals by using Proposition~\ref{propApprox}. This is illustrated in the following example.

\begin{example}[Gamma process]\label{exGamma}
  Let $(L_t)$ be a gamma process with $Y_k\sim\Gamma(\alpha\Delta,\lambda)$ for all $k=1,\dots,n$. For simplicity set $\lambda=1$. The probability density, the characteristic function and the L\'evy measure are given by 
  \begin{gather*}
    \gamma_{\alpha\Delta}(x)
    :=\frac{1}{\Gamma(\alpha\Delta)}x^{\alpha\Delta-1}e^{-x}\ind_{[0,\infty)}(x),\qquad
    \phi_\nu(u)=(1-iu)^{-\alpha\Delta}\quad\text{and}\\
    \nu(\d x)=\alpha x^{-1}e^{-x}\ind_{[0,\infty)}(x)\d x,\quad\text{for } x,u\in\R,
  \end{gather*}
  respectively. Therefore, $|\phi_\nu|$ decays with polynomial rate $\beta=\Delta\alpha$ and we can apply Lemma~\ref{lemAdjoint}. The estimation of the generalized distribution function $\chi(\nu)=\int_t^{\infty}\d\nu$ for some fixed $t>0,$ induces the gradient $\tilde\chi_\nu=\ind_{[t,\infty)}$. To approximate $\tilde\chi_\nu$ with a sequence in $\mathcal H$, we construct $\chi_n(x)=\int_{-\infty}^x(\delta_n(y-t)-\delta_n(y-n))\d y$ for a Dirac sequence $(\delta _n)$. More precisely, let $(\delta_n)\subset C^\infty(\R)$ be a family of smooth nonnegative functions satisfying $\int_{\R}\delta_n=1$ and $\supp \delta_n\subset[-1/n,1/n]$. Obviously, $(\chi_n)\subset\mathcal H$. Since $\nu$ is a finite measure on $\R\setminus(-\eps,\eps)$ for any $\eps>0$, dominated convergence shows $\|\tilde\chi_\nu-\chi_n\|_{L^2(\nu)}\to0$. Denoting the distribution function of $\Gamma(\beta,1)$ by $\Gamma_\beta$, we obtain for $\alpha\Delta<1/2$
  \begin{align}
    (A_\nu^\star)^{-1}\chi_n
    &=\Delta^{-1}\F^{-1}[(1+iu)(1+iu)^{\alpha\Delta-1}]\ast\chi_n\notag\\
    &=\Delta^{-1}\gamma_{1-\alpha\Delta}(-\cdot)\ast(\chi_n-\chi_n')\notag\\
    &=\Delta^{-1}\big(\gamma_{1-\alpha\Delta}(-\cdot)\ast\chi_n-\gamma_{1-\alpha\Delta}(-\cdot)\ast\delta_n(\cdot-t)\notag\\
    &\qquad\qquad+\gamma_{1-\alpha\Delta}(-\cdot)\ast\delta_n(\cdot-n)\big)\notag\\
    &\to\Delta^{-1}\big((1-\Gamma_{1-\alpha\Delta}(t-\cdot))-\gamma_{1-\alpha\Delta}(t-\cdot)\big)\notag\\
    &=\Delta^{-1}\gamma_{1-\alpha\Delta}(-\cdot)\ast(\ind_{(-\infty,t]}-\delta_t)=:\psi\label{eqGammaPsi}
  \end{align}
  where the convergence holds in $L^2(\R)$ owing to $\gamma_{1-\alpha\Delta}\in L^2(\R)$ for $\alpha\Delta<1/2$. Therefore, in a natural way the limiting object is $\psi=\F^{-1}[1/\phi_\nu(-\cdot)]\ast\ind_{(-\infty,t]}$. When does this limit hold in $L^2(P_\nu)$, too? As we saw above, the probability density of $P_\nu$ is bounded everywhere except for the singularity at zero which is of order $1-\alpha\Delta$. For any $t>0$ and $n$ large enough $\gamma_{1-\alpha\Delta}\ast\delta_n(\cdot-t)$ is uniformly bounded in a small neighborhood of zero such that dominated convergence around zero together with the $L^2(\R)$-convergence on the real line yields $\|(A_\nu^\star)^{-1}\chi_n-\psi\|_{L^2(P_\nu)}\to0$. Hence, Proposition~\ref{propApprox} shows $\ind_{[t,\infty)}\in\ran A_\nu^\star$. Therefore, the information bound is given by
  \[
    \int(\F[1/\phi_\nu(-\cdot)]\ast\ind_{[t,\infty)})^2\d P_\nu,
  \]
  which can be understood via definition \eqref{eqGammaPsi} or equivalently as the limit $\lim_{n\to\infty}\|(A_\nu^\star)^{-1}\chi_n\|_{L^2(P_\nu)}^2$. For $\alpha\Delta>1/2$ \citet{neumannReiss2009} show that $\nu([t,\infty))$ cannot be estimated with $\sqrt n$-rate.

\end{example}

This example shows the importance of the pseudo-locality for the devolution operator which was discussed by \citet{nicklReiss2012} in detail: If the singularity of the pointwise limit $\psi$ as in \eqref{eqGammaPsi} and the singularity of the distribution $P_\nu$ would coincide, the $L^2(P_\nu)$-norm of any approximating sequence $(A_\nu^\star)^{-1}\chi_n$ would diverge such that $\tilde\chi_\nu$ cannot be an element of $\ran A_\nu^\star$ by Proposition~\ref{propApprox}. A simple example is given by the convolution of a Gamma process and a Poisson process \citep[cf.][Sect. 3.2]{nicklReiss2012}.  

For $\delta>0$ and $l_\delta$ the largest integer which is strictly smaller than $\delta$, let $C^\delta(\R)$ denote the set of functions $f$ that possess $l_\delta$ continuous derivatives with $f^{(l_\delta)}$ being $(\delta-l_\delta)$-H\"older continuous. We will show that for suitable regularity $\delta>0$ the class
\begin{align*}
  Z^{\delta}(\R):=\Big\{\zeta=\zeta^s+\zeta^c\Big|&(1+|x|^{-1})\zeta^s(x)\in H^\delta(\R),\\ &\zeta^c\in C^{\delta+\eps}(\R)\text{ for some }\eps>0, \zeta^c(0)=0\Big\}
\end{align*}
intersected with $L^1(\nu)\cap L^2(\nu)$ is a subset of $\ran A_\nu^\star$. 
\begin{example}[Generalized distribution function]\label{exGenDist}
  Recall that the generalized distribution function of $\nu$ corresponds to the functionals $\zeta_t:=\ind_{(-\infty,t]}$ for $t<0$ and $\zeta_t:=\ind_{[t,\infty)}$ for $t>0$. It is easy to check that $\zeta_t$ can be decomposed for all $t\neq0$ in a way such that it is contained in $Z^\delta(\R)$ for any $\delta<1/2$. For instance, write $\ind_{[t,\infty)}=\zeta_t^s+\zeta_t^c$ with $\zeta_t^s(x):=e^{t-x}\ind_{[t,\infty)}(x)$ and $\zeta_t^c(x):=(1-e^{t-x})\ind_{[t,\infty)}(x)$ for $t>0$. Then $\zeta_t^s$ is a translation of the gamma density $\gamma_1$ such that its Fourier transform decays with polynomial rate one. The factor $(1+|x|^{-1})$ is harmless since $\zeta_t^s$ equals zero around the origin. Moreover, $\zeta_t^c$ is Lipschitz continuous. On the negative half line an analogous decomposition applies. 
\end{example}
For this analytic description of the range of the adjoint score operator, we apply the approach by \cite{nicklReiss2012} as well as \cite{soehlTrabs2012}. They study the deconvolution operator $\F^{-1}[1/\phi_\nu(-\cdot)]$ as Fourier multiplier on Besov spaces. We suppose that the L\'evy process $L$ with jump density $\nu$ satisfies the following.

\begin{assumption}\label{assFourierMult}\needspace{3\baselineskip}
  For some $\beta>0$ assume for all $u\in\R$
  \begin{enumerate}
    \item $|\phi_\nu(u)|\gtrsim(1+|u|)^{-\beta}$ and
    \item $x\nu$ has a bounded Lebesgue density with $|\F[x\nu](u)|\lesssim(1+|u|)^{-1}$.
  \end{enumerate}
\end{assumption}
This assumption is satisfied by the gamma process discussed in Example~\ref{exGamma} and in view of Lemma 2.1 in \cite{trabs:2011} for much larger class of L\'evy processes as well, including self-decomposable processes.
\begin{proposition}\label{propBound}
  Let the finite variation L\'evy process $(L_t)$ with $\nu\in\Theta$ satisfy Assumption~\ref{assFourierMult} for some $\beta>0$. Then for any $\zeta\in Z^\beta(\R)\cap L^1(\nu)\cap L^2(\nu)$ it holds $\zeta\in\ran A_\nu^\star$ with $(A_\nu^\star)^{-1}\zeta=\F^{-1}[1/\phi_\nu(-\cdot)]\ast\zeta$.
\end{proposition}
To prove this proposition, we combine the analysis of L\'evy processes in \cite{nicklReiss2012} with the insights on the interplay of the smoothness of $\zeta$ and the decay of the deconvolution operator in \cite{soehlTrabs2012} and with the characterization in Proposition~\ref{propApprox}. The proof is postponed to Section~\ref{secProofBound}. 

The formula $(A_\nu^\star)^{-1}\zeta=\F^{-1}[1/\phi_\nu(-\cdot)]\ast\zeta$ can be either understood in distributional sense or as the limit of an approximating sequence as illustrated in Example~\ref{exGamma}. Applying Theorem~\ref{thmScore} and Proposition~\ref{propBound} on Example~\ref{exGenDist}, we get the following convolution theorem.
\begin{theorem}\label{corInfoBoundIA}
  Let the finite variation L\'evy process $(L_t)$ with $\nu\in\Theta$ satisfy Assumption~\ref{assFourierMult} for some $\beta<1/2$. Then the limit distribution of any regular estimator of the generalized distribution function $(\R\setminus\{0\})^d\ni(t_1,\dots,t_d)\mapsto(\scapro{\zeta_{t_i}}\nu,\dots,\scapro{\zeta_{t_d}}\nu)$ is a convolution $\mathcal N(0,\Sigma)\ast M$ for some Borel probability measure $M$ and with the covariance matrix $\Sigma\in\R^{d\times d}$ given by 
  \begin{align*}
    \Sigma_{i,j}=\frac{1}{\Delta^2}\int_{\R}\Big(\mathcal F^{-1}\Big[\frac{1}{\phi_\nu(-\cdot)}\Big]\ast\zeta_{t_i}\Big)\Big(\mathcal F^{-1}\Big[\frac{1}{\phi_\nu(-\cdot)}\Big]\ast\zeta_{t_j}\Big)\d P_\nu
  \end{align*}
  for $i,j\in\{1,\dots,d\}$.
\end{theorem}
In the situations of Corollary~\ref{corInfoBoundCP} and Theorem~\ref{corInfoBoundIA} the estimator constructed by \citet{nicklReiss2012} is therefore efficient. 

\begin{remark}\label{exDecon3}
  Let us finish the considerations of Section~\ref{secDecon} for the deconvolution model $Y_j=X_j+\eps_j\sim P_\nu=\nu\ast\mu$. If the characteristic function of the error distribution $\mu$ satisfies for some $\beta>0$
  \begin{equation}\label{eqFourMultDecon}
    |\phi_\eps(u)|>0\quad\text{and}\quad|(1/\phi_\eps)'(u)|\lesssim (1+|u|)^{\beta-1},\quad u\in\R,
  \end{equation}
  Lemma 5(i) in \cite{soehlTrabs2012} shows that $\phi_\eps^{-1}$ is a Fourier multiplier on Besov spaces. Therefore, an analogous result as Proposition~\ref{propBound} applies in the deconvolution setup which can be combined with Theorem~\ref{thmDecon}. We can recover Theorem 4 in \cite{soehlTrabs2012} under weaker assumptions on the distributions of $\nu$ and $\mu$: If the distribution $\nu$ of $X_j$ possesses a Lebesgue density and if \eqref{eqFourMultDecon} is satisfied for some $\beta>0$, then the asymptotic variance of every regular estimator of the linear function $\chi(\nu)=\int\zeta\d\nu$ for some $\zeta\in Z^\delta(\R), \delta>\beta,$ is bounded from below by
  \begin{equation*}
    \int\Big(\F^{-1}\Big[\frac{1}{\phi_\eps(-\cdot)}\Big]\ast\zeta\Big)^2\d P_\nu-\Big(\int\zeta\d\nu\Big)^2.
  \end{equation*}
  As we saw in Example~\ref{exGenDist}, we need $\delta<1/2$ to apply this information bound to distribution function estimation. Therefore, we need $|\phi_\eps(u)|\gtrsim (1+|u|)^{-\beta}$ for $\beta<1/2$ which coincides with the classical condition under which the distribution function can be estimated with the parametric rate, cf. \cite{fan1991}. 
\end{remark}

\section{Extension to Banach space valued functions}\label{secFunctionals}

So far we considered $\R^d$-valued derived parameters $\chi$. The aim of the section is to generalize Theorem~\ref{thmScore} to functions $\chi:\Theta\to\mathbb B$ for a Banach space $(\mathbb B,\|\cdot\|)$. As pointed out by \citet{vanderVaart1988} for the estimation of parameters in infinite dimensional spaces efficiency means essentially efficiency for the marginals plus tightness of the limit law of the sequence of estimator.

Let $(\mathcal X_n,\mathcal A_n,P_{n,\theta}:\theta\in\Theta)$ be a locally regular indirect model at $\theta\in\Theta$ with respect to the tangent set $\dot\Theta_\theta$ and with generalized score operator $A_\theta:\lin\dot\Theta_\theta\to H_\theta$ for some Hilbert space $(H_\theta,\scapro\cdot\cdot_H)$. First, we have to generalize the notion of regularity to Banach space valued functions. The derived parameter $\chi:\Theta\to\mathbb B$ is \textit{pathwise differentiable} at $\theta\in\Theta$ with respect to the tangent set $\dot\Theta_\theta$ if for all $b\in\dot\Theta_\theta$ with associated path $[0,\tau)\ni t\mapsto \theta_t$
\[
  t^{-1}\big(\chi(\theta_t)-\chi(\theta)\big)\to\dot\chi_\theta b
\]
holds true for some continuous linear map $\dot\chi_\theta:\lin\dot\Theta_\theta\to\mathbb B$. The gradient of $\dot\chi_\theta$ is then defined as in \cite{vanderVaart1991} using the dual space $\mathbb B^\star$, which is the space of all continuous linear functions $b^\star:\mathbb B\to\R$. The composition $b^\star\circ\dot\chi_\theta:\lin\dot\Theta_\theta\to\R$ is linear and continuous and thus it can be represented by some $\tilde\chi_{\theta,b^\star}\in\overline\lin\dot\Theta_\theta$:
\[
  (b^\star\circ\dot\chi_\theta)b=\scapro{\tilde\chi_{\theta,b^\star}} b_\theta\quad\text{for all }b\in\lin\dot\Theta_\theta.
\]
Similarly, the parameter $\psi_n(P_{n,\theta})=\chi(\theta)$ is \textit{regular} if \eqref{eqDiffPsi} holds for some continuous linear map $\dot\psi_\theta:H\to\mathbb B$. The efficient influence functions $\tilde\psi_{\theta,b^\star}\in\overline\ran A_\theta$ are defined by $(b^\star\circ\dot\psi_\theta)h=\scapro{\tilde\psi_{\theta,b^\star}} h_H$ for all $h\in\ran A_\theta$.
The sequence of estimators $T_n:\mathcal X_n\to\mathbb B$ is called \textit{regular} at $\theta\in\Theta$ with respect to the rate $r_n$ if there is a fixed tight Borel probability measure $L$ on $\mathbb B$ such that for all $b\in\dot\Theta_\theta$ with corresponding submodel $t\mapsto P_{n,\theta_t}$
\[
  \frac{1}{r_n}\big(T_n-\chi(\theta_{r_n})\big)\overset{P_{n,\theta_{r_n}}}{\Longrightarrow} L,
\]
where weak convergence is defined in terms of outer probability to avoid measurability problems, that is,
\[
  \E^*_{P_{n,\theta_{r_n}}}\big[f\big(r_n^{-1}(T_n-\chi(\theta_{r_n})\big)\big]\to\int f\d L
\]
for all bounded, continuous function $f:\mathbb B\to \R$ \citep[cf. ][Def. 1.3.3]{vanderVaartWellner1996}. Now we can state the following generalization of Theorem~\ref{thmScore}.
\begin{theorem}
  Let $(\mathcal X_n,\mathcal A_n,P_{n,\theta}:\theta\in\Theta)$ be a locally regular indirect model at $\theta\in\Theta$ with respect to $\dot\Theta_\theta$ and let $\chi:\Theta\to\mathbb B$ be pathwise differentiable at $\theta$ with respect to $\dot\Theta_\theta$. Then the sequence $\psi_n:\mathcal P_n\to\mathbb B$ is regular at $\theta$ relative to $\dot\Theta_\theta$ if and only if 
  \begin{equation}\label{eqRanBanach}
    \tilde\chi_{\theta,b^\star}\in\ran A_\theta^\star\quad\text{for all}\quad b^\star\in\mathbb B^\star.
  \end{equation}
  In this case the efficient influence functions is given by the Moore--Penrose pseudoinverse $\tilde\psi_{\theta,b^\star}=(A_\theta^\star)^\dagger\tilde\chi_{\theta,b^\star}$. For any regular sequence of estimators $T_n$ the limit distribution $L$ of $r_n^{-1}(T_n-\chi(\theta_{r_n}))$ is  given by the law of a sum $N+W$ of independent, tight, Borel measurable random elements in $\mathbb B$ such that
  \begin{equation*}
    b^\star N\sim\mathcal N\big(0,\|\tilde\psi_{\theta,b^\star}\|_H^2\big)=\mathcal N\big(0,\|(A_\theta^\star)^\dagger\tilde\chi_{\theta,b^\star}\|_H^2\big).
  \end{equation*}
\end{theorem}
The proof of this Theorem is analogous to Theorem~\ref{thmScore} with an additional application of Lemma A.2 in \cite{vanderVaart1991}. We omit the details.
\begin{remark}
  The type of regularity which we used for the parameters $\psi_n(P_{n,\theta})$ to apply the convolution theorem is quite strong because the derivative $\dot\psi_\theta$ has to be continuous with respect to the norm topology of $\mathbb B$. Necessarily, the range condition \eqref{eqRanBanach} has to hold for all $b^\star\in\mathbb B^\star$ which may fail if the dual space is large. This problem can be solved by using a weaker topology on $\mathbb B$ which is generated by a subspace $B'\subset\mathbb B^\star$ as shown by \citet[Sect. 3]{vanderVaart1988}.
\end{remark}

To show tightness of the limit distribution may be a difficult problem for inverse problems. In the i.i.d. setting the theory of smoothed empirical processes by \citet{GineNickl2008} turns out to be useful as the following example shows.
\begin{example}
  Let us consider again the estimation of the distribution function $\R\ni t\mapsto\nu((-\infty,t])$ in the deconvolution model in Section~\ref{secDecon}. On the whole real line the parameter $\chi(\nu)=(\nu((-\infty,t]))_{t\in\R}$ maps into the space of bounded functions $\mathbb B=\ell^\infty(\R)$ which is equipped with the supremums norm. Under suitable conditions \citet{soehlTrabs2012} have shown a uniform central limit theorem in $\ell^\infty(\R)$ for the canonical plug-in kernel estimator. In particular, their limit distribution is tight and the finite dimensional distributions coincide the the information bound from Theorem~\ref{thmDecon} and Remark~\ref{exDecon3}. In the L\'evy model \citet{nicklReiss2012} have proved a Donsker theorem for the estimation of the generalized distribution function in $\ell^\infty((-\infty,-\delta]\cup[\delta,\infty))$ for $\delta>0$ which proves tightness of the limit process in Theorem~\ref{corInfoBoundIA}.
\end{example}

\section{Remaining proofs}\label{secProofs}
\subsection{Proof of Lemma~\ref{lemDensity}}\label{secProofDens}
As discussed above the law of the white noise $\eps\dot W=y_{\eps,0}$ is a symmetric (zero mean) Gaussian measure on $(E,\mathscr B(E))$. Its (unique) reproducing kernel Hilbert space is $Y$ with norm $\|\cdot\|_\eps:=\eps^{-1}\|\cdot\|_Y$. To see this, note that every functional $\phi\in E^\star$ can be represented by $\phi=\scapro y\cdot_Y=\scapro {\tilde y}\cdot_\eps$ for some $y\in Y$ and $\tilde y=\eps^2y$. Then 
\[
  \phi(\eps\dot W)\sim \mathcal N(0,\|\tilde y\|_\eps^2)=\mathcal N(0,\eps^2\|y\|_Y^2).
\]
The Cameron--Martin formula \citep[Prop. 2.24]{daPratoZabczyk1992} yields that $P_{\eps,x}$ and $P_{\eps,0}$ are equivalent measures on $(E,\mathscr B(E))$ with Radon-Nikodym derivative 
\begin{align*}
  \frac{\d P_{\eps,x}}{\d P_{\eps,0}}(y_{\eps,0})=\exp\Big(\scapro{y_{\eps,0}}{K (x)}_\eps-\frac{1}{2}\|K (x)\|_\eps^2\Big)\quad P_{\eps,0}-a.s.
\end{align*}
and thus
\begin{align*}
  \frac{\d P_{\eps,0}}{\d P_{\eps,\theta}}(y_{\eps,\theta})&=\exp\Big(-\scapro{y_{\eps,\theta}}{K(\theta)}_\eps+\frac{1}{2}\|K(\theta)\|_\eps^2\Big)\quad\text{and}\\
  \frac{\d P_{\eps,x}}{\d P_{\eps,\theta}}(y_{\eps,\theta})
  &=\frac{\d P_{\eps,x}}{\d P_{\eps,0}}(y_{\eps,\theta})\frac{\d P_{\eps,0}}{\d P_{\eps,\theta}}(y_{\eps,\theta})\\
  &=\exp\Big(\scapro{\eps \dot W}{K(x)-K(\theta)}_\eps-\frac{1}{2}\|K(x)-K(\theta)\|_\eps^2\Big)\\
  &=\exp\Big(\Big\langle\dot W,\frac{K(x)-K(\theta)}{\eps}\Big\rangle_Y-\frac{1}{2}\Big\|\frac{K(x)-K(\theta)}{\eps}\Big\|_Y^2\Big)\quad P_{\eps,\theta}-a.s.\qedhere
\end{align*}

\subsection{Proof of Proposition~\ref{propRegularity}}\label{secProofReg}
  Without loss of generality we assume $\Delta=1$ in this and the following subsections.

  For any $b\in\dot\Theta_\nu$ let $P_\nu\ast (b\nu)=(P_\nu\ast (b\nu))^a+(P_\nu\ast (b\nu))^\perp$ be the Lebesgue decomposition of $P_\nu\ast (b\nu)$ with respect to $P_\nu$, that is the first and second measure are absolutely continuous and singular with respect to $P_\nu$, respectively. $\frac{\d P_\nu\ast(b\nu)}{\d P_\nu}$ is then defined as the Radon--Nikodym density of $(P_\nu\ast (b\nu))^a$ with respect to $P_\nu$. Therefore, $A_\nu b=\frac{\d P_\nu\ast(b\nu)}{\d P_\nu}-\int b\d\nu$ is well-defined without further assumptions. 
  
  In a first step we will show $L^2$-differentiability of the submodels corresponding to some direction $b\in\dot\Theta_\theta\cap L^\infty(\nu)$. The extension to the whole tangent set is proved in the second step. In the last step we will see that even $P_\nu\ast(b\nu)\ll P_\nu$ holds true.

  \textit{Step 1:}
  Let $b\in L^1(\nu)\cap L^\infty(\nu)$. We will show that the associated model $[0,1)\ni t\mapsto P_{\nu_t}$ is $L^2$-differentiable at 0 with derivative $A_\nu b\in L^2_0(P_\nu)$ as defined in \eqref{eqScoreOp}, that is 
  \begin{equation}\label{eqRegularity}
    \int\Big(\frac{\d P_{\nu_t}-\d P_\nu}{t\d P_\nu}-A_\nu b\Big)^2\d P_\nu\to0\quad\text{as}\quad t\to0.    
  \end{equation}
  Note that $\frac{\d P_{\nu_t}}{\d P_\nu}\in L^2(P_\nu)$ by \eqref{eqH2dist}, which follows by a similar argument as the one following \eqref{eqL2Diff}. Applying Proposition 1.199 in \cite{witting1985}, the $L^2$-regularity \eqref{eqRegularity} implies the proposed Hellinger differentiability for all $b\in\dot\Theta_\nu\cap L^\infty(\nu)$.
  
  Defining the measure $\nu_t^c$ via the density$\frac{\d\nu_t^c}{\d\nu}=(k(tb)-1)=:f_{\nu_t^c}$, we write as a consequence of \eqref{eqDens} and Remark~27.3 in \cite{sato1999} 
  \[
    P_{\nu_t}=e^{-\nu_t^c(\R)}\sum_{k=0}^\infty \frac{1}{k!}(\nu_t^c)^{\ast k}\ast P_\nu
  \]
  Arranging terms
  \begin{align*}
  &\int\Big(\frac{\d P_{\nu_t}-\d P_\nu}{t\d P_\nu}-A_\nu b\Big)^2\d P_\nu\\
  =&\int\Big(\frac{\d P_{\nu_t}-\d P_\nu-t\d (P_\nu\ast(b\nu))+t\big(\int b\nu\big)\d P_\nu}{t\d P_{\nu}}\Big)^2\d P_\nu\\
  =&\int\bigg(\frac{(e^{-\nu_t^c(\R)}-1+t\int b\d\nu)\d P_\nu+\d\big(((e^{-\nu_t^c(\R)}f_{\nu_t^c}-tb)\nu)\ast P_\nu\big)}{t \d P_\nu}\\
  &\qquad\quad+\frac{e^{-\nu_t^c(\R)}\sum_{k=2}^\infty(k!)^{-1}\d\big((\nu_t^c)^{\ast k}\ast P_\nu\big)}{t \d P_\nu}\bigg)^2\d P_\nu,
  \end{align*}
  all three terms in the numerator turn out to be of order $t^2$. Note that we can dominate $|f_{\nu_t^c}|\le t|b|$ as well as $|f_{\nu_t^c}-tb|\le t^2|b^2|$ by $|k(y)-1|\le|y|$ and $|k(y)-1-y|\le y^2$. Therefore,
  \begin{align*}
    \Big|e^{-\nu_t^c(\R)}-1+t\int b\d\nu\Big|
    &=\Big|\sum_{k\ge2}\frac{(-\nu_t^c(\R))^k}{k!}-\Big(\nu_t^c(\R)-t\int b\d\nu\Big)\Big|\\
    &\le\sum_{k\ge2}\frac{t^k\|b\|_{L^1(\nu)}^k}{k!}+t^2\|b\|_{L^2(\nu)}^2\\
    &\le \big(e^{\|b\|_{L^1(\nu)}}+\|b\|_{L^2(\nu)}^2\big)t^2,\\
    \big|e^{-\nu_t^c(\R)}f_{\nu_t^c}-tb\big|
    &=\Big|\big(e^{-\nu_t^c(\R)}-1\big)f_{\nu_t^c}+f_{\nu_t^c}-tb\Big|\\
    &\le t^2\big(e^{\|b\|_{L^1(\nu)}}|b|+|b|^2\big),\\
    \Big|e^{-\nu_t^c(\R)}\sum_{k=2}^\infty\frac{(\nu_t^c)^{\ast k}}{k!}\Big|
    &\le t^2\sum_{k=2}^\infty\frac{(|b|\nu)^{\ast k}}{k!}.
  \end{align*}
  Hence, we estimate
  \begin{align}
  &\int\Big(\frac{\d P_{\nu_t}-\d P_\nu}{t\d P_\nu}-A_\nu b\Big)^2\d P_\nu\notag\\
  \le&t^2\int\bigg(\frac{(e^{\|b\|_{L^1(\nu)}}+\|b\|_{L^2(\nu)}^2)\d P_\nu+\big(e^{\|b\|_{L^1(\nu)}}|b|+|b^2|\big)\d(\nu\ast P_\nu)}{\d P_\nu}\notag\\
  &\qquad\quad+\frac{\sum_{k=2}^\infty(k!)^{-1}\d\big((|b|\nu)^{\ast k}\ast P_\nu\big)}{\d P_\nu}\bigg)^2\d P_\nu\notag\\
  \le&t^2\big(e^{\|b\|_{L^1(\nu)}}+\|b\|_{L^2(\nu)}^2\big)^2\int\bigg(\frac{\d\Big(\sum_{k=0}^\infty((|b|+|b|^2)\nu)^{\ast k}/(k!)\Big)\ast P_\nu}{\d P_\nu}\bigg)^2\d P_\nu.\notag
  \end{align}
  Introducing an infinite divisible distribution $\mu$ without diffusion component, without drift and with finite jump measure $(|b|+|b|^2)\nu$, the previous line can be written as
  \begin{align}
  &t^2\underbrace{\big(e^{\|b\|_{L^1(\nu)}}+\|b\|_{L^2(\nu)}^2\big)^2e^{2\|b\|_{L^1(\nu)}+2\|b\|_{L^2(\nu)}^2}}_{=:C(\|b\|_{L^1(\nu)},\|b\|_{L^2(\nu)})}\int\left(\frac{\d(\mu\ast P_\nu)}{\d P_\nu}\right)^2\d P_\nu\label{eqL2Diff}
  \end{align}
  Therefore, the assertion holds true provided the Hellinger integral $H_2(\mu\ast P_\nu,P_\nu)=\int\left(\d(\mu\ast P_\nu)/\d P_\nu\right)^2\d P_\nu$ is finite. To show this, we apply the bound of Renyi's distance $R_2$ for infinite divisible distributions by \citet[Thm. 2.6]{Liese1987}. Using that both distributions, $\mu\ast P_\nu$ and $P_\nu$, have the same drift and have finite variation, we obtain (for $\alpha=2$)
  \begin{equation*}
    \tfrac{1}{2}\log H_2(\mu\ast P_\nu,P_\nu)=R_2(\mu\ast P_\nu,P_\nu)\le\tfrac{1}{2}\chi^2\big((|b|+b^2+1)\nu,\nu\big)
  \end{equation*}
  where the $\chi^2$-distance of the jump measures satisfies
  \begin{align*}
    \chi^2\big((|b|+b^2+1)\nu,\nu\big):=&\frac{1}{2}\int\left( \frac{\d((|b|+b^2+1)\nu)}{\d\nu}-1\right)^2\d \nu\\
    =&\frac{1}{2}\int\left(|b|+b^2\right)^2\d\nu
    \le\frac{1}{2}\big(1+\|b\|_{L^\infty(\nu)}^2\big)\|b\|_{L^2(\nu)}^2<\infty.
  \end{align*}
  The combination with the bound \eqref{eqL2Diff} yields
  \[
    \int\Big(\frac{\d P_{\nu_t}-\d P_\nu}{t\d P_\nu}-A_\nu b\Big)^2\d P_\nu
    \le t^2C\big(\|b\|_{L^1(\nu)},\|b\|_{L^2(\nu)}\big)e^{\frac{1}{2}(1+\|b\|_{L^\infty(\nu)}^2)\|b\|_{L^2(\nu)}^2}.
  \]
  As $t\to0$ this upper bound converges to zero which shows the $L^2$-differentiability. We conclude $\int A_\nu b\d P_\nu=0$ for all $b\in\dot\Theta_\nu\cap L^\infty(\nu)$.
  
  \textit{Step 2:}
  To show continuity of $A_\nu|_{L^1(\nu)\cap L^\infty(\nu)}$, let $\eps>0$ and $b\in L^1(\nu)\cap L^\infty(\nu)$ with $\|b\|_{L^2(\nu)}^2<\eps$. By \eqref{eqRegularity}, $\frac{1}{2}A_\nu b$ is the $L^2$-limit of $t^{-1}(\sqrt{\d P_{\nu_t}}-\sqrt{\d P_\nu})$ and thus for $t$ small enough
  \[
    \|A_\nu b\|_{L^2(P_\nu)}^2\le \frac{2}{t^2} \int\big(\sqrt{\d P_{\nu_t}}-\sqrt{\d P_\nu}\big)^2 +\eps.
  \]
  As above Theorem 2.6 in \cite{Liese1987} for $\alpha=1/2$ yields the estimate for the Hellinger distance of the infinite divisible distributions
  \begin{align}
    \int\big(\sqrt{\d P_{\nu_t}}-\sqrt{\d P_\nu}\big)^2
    &\le2\Big(1-\exp\Big(-2\int(\sqrt{\d \nu_t}-\sqrt{\d \nu})^2\Big)\Big)\label{eqHellingDist}\\
    &=2\Big(1-\exp\Big(-2\int(\sqrt{k(tb)}-1)^2\d\nu\Big)\Big).\notag
  \end{align}
  Since $|\sqrt{k(y)}-1|\le|(\sqrt{k(y)}+1)(\sqrt{k(y)}-1)|=|k(y)-1|\le|y|$ and $1-e^{-y}\le|y|$ for $y\in\R$, the previous display can be bounded by
  \[
    2\Big(1-\exp\Big(-2\int(tb)^2\d\nu\Big)\Big)
    \le4t^2\|b\|_{L^2(\nu)}^2.
  \]
  Because $\eps>0$ was arbitrary, we conclude $\|A_\nu b\|_{L^2(P_\nu)}\lesssim\|b\|_{L^2(\nu)}$ which is equivalent to the continuity of the linear operator $A_\nu|_{L^1(\nu)\cap L^\infty(\nu)}$. Since $L^1(\nu)\cap L^\infty(\nu)$ is dense in $\dot\Theta_\nu$, there is a unique continuous extension $A_\nu$ on $\dot\Theta_\nu$ satisfying $A_\nu b=\frac{\d P_\nu\ast(b\nu)-\int b\nu\d P_\nu}{\d P_\nu}$ for all $b\in L^1(\nu)\cap L^\infty(\nu)$. 
  
  Now, for any $b\in\dot\Theta_\nu$ with associated path $t\mapsto\d \nu_t=k(tb)\d \nu$ and for any positive null sequence $(t_m)_{m\in\N}$ and let $\tilde b_m\in L^1(\nu)\cap L^\infty(\nu)$ with path $t\mapsto \d \tilde \nu_t:=k(t\tilde b_m)\d\nu$ such that $\|b-\tilde b_m\|_{L^2(\nu)}\to 0$ and $\|\tilde b_m\|_{L^\infty(\nu)}=o(|\log t_m|^{1/2})$ as $m\to\infty$. Then
  \begin{align}
   &t_m^{-2}\int\bigg(\sqrt{\frac{\d P_{\nu_{t_m}}}{\d P_\nu}}-1-\frac{t_m}{2}A_\nu b\bigg)^2\d P_\nu\notag\\
   \le&\frac{3}{t_m^2}\int\Big(\sqrt{\frac{\d P_{\nu_{t_m}}}{\d P_\nu}}-\sqrt{\frac{\d P_{\tilde\nu_{t_m}}}{\d P_\nu}}\Big)^2\d P_\nu
   +\frac{3}{t_m^2}\int\Big(\sqrt{\frac{\d P_{\tilde\nu_{t_m}}}{\d P_\nu}}-1-\frac{t_m}{2}A_\nu \tilde b_m\Big)^2\d P_\nu\notag\\
   &\qquad+\frac{3}{4}\int\Big(A_\nu \tilde b_m-A_\nu b\Big)^2\d P_\nu\label{eqLANdecomp}
  \end{align}
  The frist term is the Hellinger distance between $P_{\nu_t}$ and $P_{\tilde\nu_t}$ which can be bounded as in \eqref{eqHellingDist}
  \begin{align*}
    &t_m^{-2}\int\bigg(\sqrt{\frac{\d P_{\nu_{t_m}}}{\d P_\nu}}-\sqrt{\frac{\d P_{\tilde\nu_{t_m}}}{\d P_\nu}}\bigg)^2\d P_\nu\\
    &\qquad\le2t_m^{-2}\Big(1-\exp\Big(-2\int\big(\sqrt{\d\nu_{t_m}}-\sqrt{\d \tilde\nu_{t_m}}\big)^2\Big)\Big)\\
    &\qquad=2t_m^{-2}\Big(1-\exp\Big(-2\int\Big(\sqrt{k(t_mb(x))}-\sqrt{k(t_m\tilde b_m(x))}\Big)^2\nu(\d x)\Big)\Big).
  \end{align*}
  An easy calculation shows $|(\sqrt{k})'(x)|\le1$ for all $x\in\R$ and thus the above display can be bounded by the mean value theorem
  \begin{align*}
    2t_m^{-2}\Big(1-\exp\big(-2t_m^2\|b-\tilde b_m\|_{L^2(\nu)}^2\big)\Big)
    \le4\|b-\tilde b_m\|_{L^2(\nu)}^2\to0.
  \end{align*}
  The second term in \eqref{eqLANdecomp} converges to zero according to Step 1 provided $\|\tilde b\|_{L^\infty(\nu)}=o(|\log t|^{1/2})$. Applying continuity of $A_\nu$, the third term in \eqref{eqLANdecomp} vanishes as well. Therefore, we have shown that $A_\nu b$ is the Hellinger derivative of $P_{\nu_t}$ for all $b\in\dot\Theta_\nu$. 
  
  \textit{Step 3:}
  Finally, we will show $P_\nu\ast(b\nu)\ll P_\nu$ for all $b\in\dot\Theta_\nu\cap L^\infty(\nu)$. By construction $|b|\in\dot\Theta_\nu\cap L^\infty(\nu)$, too. Let $P_\nu\ast (|b|\nu)=(P_\nu\ast (|b|\nu))^a+(P_\nu\ast (|b|\nu))^\perp$ be Lebesgue's decomposition with respect to $P_\nu$ where both measures can be chosen to be nonnegative and finite. According to Step 1, $\int A_\nu |b|\d P_\nu=0$ which yields together with the nonnegativity of the measures and Fubini's theorem
  \[
    \int |b|\d\nu=\int\frac{\d (P_\nu\ast(|b|\nu))}{\d P_\nu}\d P_\nu
    =\int\d(P_\nu\ast (|b|\nu))^a
    \le\int\d P_\nu\ast(|b|\nu)
    =\int|b|\d\nu.
  \]
  We conclude $(P_\nu\ast (|b|\nu))^a=P_\nu\ast(|b|\nu)$ or equivalently $P_\nu\ast (|b|\nu)\ll P_\nu$. Now for any event $A\in\mathscr B(\R)$ with $P_\nu\ast(|b|\nu)(A)=0$ we have
  \begin{align*}
    |P_\nu\ast(b\nu)(A)|&=\Big|\int_{\R^2}\ind_A(x+y)b(x)\nu(\d x)P_\nu(\d y)\Big|\\
    &\le\int_{\R^2}\ind_A(x+y)|b(x)|\nu(\d x)P_\nu(\d y)
    =P_\nu\ast(|b|\nu)(A)=0.
  \end{align*}
  Consequently, $P_\nu\ast(b\nu)\ll P_\nu\ast(|b|\nu)\ll P_\nu$.
\qed

\subsection{Proof of Lemma~\ref{lemCPadjoint}}\label{secProofCPajoint}
  First, we show $\nu^{\ast l}\ast P_\nu\ll P_\nu$ for any $l\in\N$. Let $A\in\mathscr B(\R)$ satisfy $P_\nu(A)=0$. \eqref{eqConExp} yields
  \[
    0=e^{-\Delta\lambda}\sum_{k=0}^\infty\frac{\Delta^k}{k!}\int\ind_A(x+\Delta\gamma)\d\nu^{\ast k}(\d x)
  \]
  and thus $\int\ind_A(x+\Delta\gamma)\d\nu^{\ast k}(\d x)=0$ for all $k\in\N$. But this implies by linearity of the convolution that $\nu^{\ast l}\ast P_\nu(A)=0$.

  To see that $A_\nu^\star$ is well defined on equivalence classes with respect to $P_\nu$ zero sets, note that $\nu\ast P_\nu\ll P_\nu$ implies that for any function $g$ with $g(x)=0$ for $P_\nu$-a.e. $x\in\R$ it holds $P_\nu(-\cdot)\ast g(y)=0$ for $\nu$-a.e. $y\in\R$. It remains to show $A_\nu^\star g\in\mathcal H$ for $g\in \mathcal G$. For any $g\in L^\infty(P_\nu)$ there is a set $A\in\mathscr B(\R)$ with $P_\nu(A)=0$ such that $g(y)\le C$ for some constant $C>0$ and for all $y\notin A$. Using
  \[
    0=\nu^{\ast l}\ast P_\nu(A)=\int\int\ind_{A-\{x\}}(y)P_\nu(\d y)\nu^{\ast l}(\d x)
  \]
  we infer $P_\nu(A-\{x\})=0$ for $\nu^{\ast l}$-a.e. $x\in\R$ and therefore $P_\nu(-\cdot)\ast g(y)=\int g(x+y)P_\nu(\d x)$ is bounded by $C$ for $\nu^{\ast l}$-a.e. $y\in\R$. Hence, $\|P_\nu(-\cdot)\ast g\|_{L^\infty(\nu^{\ast l})}\le C$ for any $l\in\N$.
  \qed

\subsection{Proof of Lemma~\ref{lemAdjoint}}\label{secProofAdjoint}
  \textit{(i)} We will determine the adjoint score operator and its inverse on the subsets $\mathcal G$ and $\mathcal H$ as defined in \eqref{eqSubsetsIA}. In the case of infinite jump activity the application of Fubini's theorem in \eqref{eqFubini} holds as well. Hence, the adjoint score operator on $\mathcal G$ is given by $A_\nu^\star g=P_\nu(-\cdot)\ast g$. To verify that $A_\nu^\star|_\mathcal G$ is well-defined, we note first that by the Sobolev embedding any $g\in\mathcal G$ has a version in $C^1(\R)$. Throughout we can identify $g$ with this smooth version. Then, we obtain  $A_\nu^\star g(0)=\int g\d P_\nu=0$ and $\|(A_\nu^\star g)^{(l)}\|_\infty=\|P_\nu(-\cdot)\ast(g^{(l)})\|_\infty\le\|g^{(l)}\|_\infty\le\|g\|_{C^1}$ for $l=0,1$. Hence, $A_\nu^\star g$ is a bounded function and
  \begin{align*}
    \int|A_\nu^\star g(x)|\nu(x)\d x&\le\int\big(\|A_\nu^\star g\|_\infty\wedge(|x|\|(A_\nu^\star g)'\|_\infty)\big)\d\nu(x)\\
    &\le\|g\|_{C^1}\int(1\wedge|x|)\d\nu(x).
  \end{align*}
  A similar estimate holds for $L^2(\nu)$. Therefore, $A_\nu^\star g\in L^1(\nu)\cap L^\infty(\R)\subset\dot\Theta_\nu$. Owing to $\|\phi_\nu\|_\infty\le1$, it holds for any $s>0$
  \[
    \big\|(1+|u|^2)^{s/2}\F[A_\nu^\star g](u)\big\|_{L^2}
    =\big\|(1+|u|^2)^{s/2}\phi_\nu(-u)\F g(u)\big\|_{L^2}
    \le\|g\|_{H^s}<\infty.
  \]
  We conclude $\ran A_\nu^\star|_{\mathcal G}\subset\mathcal H$. 

  Let us show now that the inverse adjoint score operator as given in \eqref{eqAdjointOpInv} is well-defined on $\mathcal H$. Applying the assumption $|\phi_\nu(u)|\gtrsim(1+|u|)^{-\beta}$, we obtain for all $b\in\mathcal H$ and $s>0$
  \begin{align*}
    \big\|(1+|u|^2)^{s/2}\F[(A_\nu^\star)^{-1} b](u)\big\|_{L^2}
    &=\big\|(1+|u|^2)^{s/2}\F b(u)/\phi_\nu(-u)\big\|_{L^2}\\
    &\lesssim\|(1+|u|^2)^{(s+\beta)/2}\F b(u)\big\|_{L^2}\\
    &\le\|b\|_{H^{s+\beta}}<\infty.
  \end{align*}
  Therefore, $(A_\nu^\star)^{-1}b\in H^\infty(\R)$ and the Sobolev embedding yields $\|(A_\nu^\star)^{-1}b\|_{L^2(P_\nu)}\le\|(A_\nu^\star)^{-1}b\|_\infty\le\|(A_\nu^\star)^{-1}b\|_{C^s}<\infty$. It remains to verify the condition $\int(A_\nu^\star)^{-1}b\d P_\nu=0$. By construction
  \begin{align*}
    \int (A_\nu^\star)^{-1}b\d P_\nu
    &=\big(((A_\nu^\star)^{-1}b)\ast P_\nu\big)(0)\\
    &=\big(\F^{-1}\Big[\F b/\phi_\nu(-\cdot)\Big]\ast P_\nu\big)(0)
    = b(0)
  \end{align*}
  where the last equality is clear in distributional sense and can be shown via integration against test functions. Since $b(0)=0$ for all $b\in\mathcal H$, we conclude $\ran (A_\nu^\star)^{-1}|_{\mathcal H}\subset\mathcal G$. 

  By construction $g=(A_\nu^\star)^{-1}A_\nu^\star g$ and $b=A_\nu^\star(A_\nu^\star)^{-1}b$ for all $g\in\mathcal G,b\in\mathcal H$ which proves that $A_\nu^\star|_\mathcal G$ is a bijection from $\mathcal G$ onto $\mathcal H$.
  
  \textit{(ii)} Let us show that $\mathcal G$ is dense in $L_0^2(P_\nu)$. Since the Borel $\sigma$-field is generated by $\mathcal E:=\{[a,b]:-\infty<a<b<\infty\}$, the set of indicator functions $\{\mathbbm 1_E:E\in\mathcal E\}$ is dense in $L^2(P_\nu)$. Hence, it suffices to approximate in $L^2(P_\nu)$-sense the indicators $\mathbbm 1_E, E\in\mathcal E$, by smooth $L^2(\R)$-integrable functions. Let $\eps>0$ be arbitrary and let us denote the distribution function of $P_\nu$ by $F(x):=P_\nu((-\infty,x]), x\in\R$. Since $F$ is right continuous with left limits, for all $a<b$ there are $a'<a, b<b'$ such that
  \[
    P_\nu\big((a',b']\setminus[a,b]\big)=P_\nu\big((a',a)\big)+P_\nu\big((b,b']\big)=F(a-)-F(a')+F(b')-F(b)<\eps.
  \]
  Therefore, for any $A\in\mathcal E$ there is a bounded set $B\in\mathscr B(\R)$ satisfying $A\subset B$, $P_\nu(B\setminus A)<\eps$ and the distance between $x\in A$ and $\R\setminus B$ is strictly positive. Consequently, there is some non-negative $\psi\in C^\infty(\R)$ with $\psi(x)=1$ for $x\in A$, $\|\psi\|_\infty<1$ and $\supp\psi\subset B$. Obviously, $\psi$ is contained in $H^\infty(\R)\cap L^2(P_\nu)$ and $\|\mathbbm 1_A-\psi\|_{L^2(P_\nu)}<\sqrt\eps$. Since $L_0^2(P_\nu)=(\lin 1)^{\perp}$ is a closed subspace of $L^2(P_\nu)$, we conclude that $\mathcal G$ is dense in $L_0^2(P_\nu)$. 

  Continuity of $A_\nu^\star$ follows from the continuity of $A_\nu$ which was shown in Proposition~\ref{propRegularity}. Hence, $A_\nu^\star$ is uniquely given by the continuous extension of $A_\nu^\star|_\mathcal G$ to $L_0^2(P_\nu)$.
  \qed
\subsection{Proof of Proposition~\ref{propBound}}\label{secProofBound}
  Taking the derivative of the L\'evy-Khintchine formula~\eqref{eqLevyKhintchine}, we obtain
  \begin{align*}
    \phi_\nu'(u)=\phi_\nu(u)\big(i\gamma+\F[x\nu](u)\big),\quad u\in\R.
  \end{align*}
  In a first step we will show that the drift can be discarded, which was also the case for the upper bound in \cite{nicklReiss2012}. Since Lemma~\ref{lemAdjoint} shows that the inverse adjoint score operator is given by $\F^{-1}[1/\phi_\nu(-\cdot)]$ on the smooth subset $\mathcal G$, we study the mapping properties of this deconvolution operator in Step~2. Finally in Step~3, we apply the characterization in Proposition~\ref{propApprox} to prove that $Z^\beta(\R)\subset \ran A_\nu^\star$ and to determine $(A_\nu^\star)^{-1}$ on $Z^\beta(\R)$.
  
  \textit{Step 1:}
  Let us show that $\gamma=0$ can be assumed, meaning that the process $L$ has no drift. For any $\gamma\in\R$ consider the infinitely divisible distribution $\tilde P_\nu:=P_\nu\ast\delta_{-\gamma}$.
  Then the following map is an isomorphism
  \[
    \Phi:\quad L_0^2(P_\nu)\to L_0^2(\tilde P_\nu),\quad g\mapsto g(\cdot+\gamma).
  \]
  Lemma~\ref{lemAdjoint} determines the adjoint score operator $\tilde A_\nu^\star$ which corresponds to $\tilde P_\nu$. Also by Lemma~\ref{lemAdjoint} we see for $g\in\mathcal G$ that $A_\nu^\star g=P_\nu(-\cdot)\ast g=\tilde P_\nu(-\cdot)\ast g(\cdot+\gamma)$. Therefore, $A_\nu^\star=\tilde A_\nu^\star\circ\Phi$ which implies
  \[
    \ran A_\nu^\star=\ran\big(\tilde A_\nu^\star\circ\Phi\big)=\ran\tilde A_\nu^\star.
  \]
  Hence, for the rest of the proof suppose $\gamma=0$.

  \textit{Step 2:}
  The aim of this step is to show for $\zeta=\zeta^s+\zeta^c$ and any $\eps>0$
  \begin{equation}\label{eqBoundNorm}
    \|\F^{-1}[\phi_\nu^{-1}(-\cdot)\F\zeta]\|_{L^2(P_\nu)}\lesssim \big\|\zeta^s(x)\big\|_{H^{\beta}}+\big\|\tfrac{1}{x}\zeta^s(x)\big\|_{H^{\beta}}+\|\zeta^c\|_{C^{\beta+\eps}}.    
  \end{equation}
  To this end note that Assumption~\ref{assFourierMult} yields, due to $\gamma=0$,
  \[
    |\phi_\nu^{-1}(u)|\lesssim(1+|u|)^{\beta}\quad\text{and}\quad|(\phi_\nu^{-1})'(u)|\lesssim(1+|u|)^{\beta-1}
  \]
  and thus Lemma 4(c) in \cite{nicklReiss2012} or alternatively Lemma 5(i) in \cite{soehlTrabs2012} shows that for all $s\in\R, p,q\in[1,\infty]$ the linear map
  \begin{align}\label{eqFourMult}
    B_{p,q}^{s+\beta}(\R)\to B_{p,q}^{s}(\R),\quad 
    f\mapsto\F^{-1}[\phi_\nu^{-1}(-\cdot)\F f]
  \end{align}
  is bounded. This yields for any $\eps>0$ and $\zeta^c\in C^{\beta+\eps}(\R)$
  \begin{equation}\label{eqContPart}
    \|\F^{-1}[\phi_\nu^{-1}(-\cdot)\F\zeta^c]\|_{L^2(P_\nu)}\lesssim\|F^{-1}[\phi_\nu^{-1}(-\cdot)\F\zeta^c]\|_{\infty}
    \lesssim\|\zeta^c\|_{B^\beta_{\infty,1}}
    \lesssim\|\zeta^c\|_{C^{\beta+\eps}}.
  \end{equation}
  For the singular part we apply a similar decomposition in \cite{nicklReiss2012}. Integration by parts yields
  \begin{align}
    \F^{-1}\Big[\frac{\F\zeta^s}{\phi_\nu(-\cdot)}\Big]
    &=\F^{-1}\Big[\frac{(\F[\frac{1}{ix}\zeta^s(x)])'}{\phi_\nu(-\cdot)}\Big]\notag\\
    &=ix\F^{-1}\Big[\frac{\F[\frac{1}{ix}\zeta^s(x)]}{\phi_\nu(-\cdot)}\Big]+\F^{-1}\Big[\F\big[\tfrac{1}{ix}\zeta^s(x)\big](\phi_\nu^{-1})'(-\cdot)\Big]\label{eqDecomp}
  \end{align}
  Note that $1/\phi_\nu$ is a Fourier multiplier from $H^\beta$ into $H^{0}=L^2(\R)$ on the assumption $|\phi_\nu(u)|\gtrsim(1+|u|)^{-\beta}$. Similarly, $(\phi_\nu^{-1})'$ is a Fourier multiplier from $H^\beta$ into $H^{1}$. Hence, for $\frac{1}{i\cdot}\zeta^s\in H^\beta$
  \begin{align*}
    \F^{-1}\Big[\frac{\F[\frac{1}{ix}\zeta^s(x)]}{\phi_\nu(-\cdot)}\Big]
    &\in L^2(\R),\\
    \F^{-1}\Big[\F\big[\tfrac{1}{ix}\zeta^s(x)\big](\phi_\nu^{-1})'(-\cdot)\Big]
    &\in H^{1}(\R)\subset C^0(\R),
  \end{align*}
  where the last inclusion holds by the Sobolev embedding. Moreover, 
  \begin{align*}
    ix\F^{-1}\Big[\frac{\F[\frac{1}{ix}\zeta^s(x)]}{\phi_\nu(-\cdot)}\Big]
    &=\F^{-1}\Big[\frac{\F[\zeta^s(x)]}{\phi_\nu(-\cdot)}\Big]-\F^{-1}\Big[\F\big[\tfrac{1}{ix}\zeta^s(x)\big](\phi_\nu^{-1})'(-\cdot)\Big]
  \end{align*}
  which is an $L^2(\R)$-function. Applying Lemma 4(a) from \cite{nicklReiss2012}, the distribution $xP_\nu(\d x)$ has a bounded Lebesgue density and which yields together with $|x|^2\lesssim|x||1+ix|^2$ and the continuous embeddings above
  \begin{align*}
    \|\F^{-1}[\phi_\nu^{-1}(-\cdot)\F\zeta^s]\|_{L^2(P_\nu)}
    \le&\int_{\R}\Big|ix\F^{-1}\Big[\frac{\F[\frac{1}{ix}\zeta^s(x)]}{\phi_\nu(-\cdot)}\Big](x)\Big|^2\d P_\nu(x)\\
    &+\int_{\R}\Big|\F^{-1}\Big[\F\big[\tfrac{1}{ix}\zeta^s(x)\big](\phi_\nu^{-1})'(-\cdot)\Big](x)\Big|^2\d P_\nu(x)\\
    \lesssim&\Big\|(1+ix)\F^{-1}\Big[\frac{\F[\frac{1}{ix}\zeta^s(x)]}{\phi_\nu(-\cdot)}\Big](x)\Big\|_{L^2}\\
    &+\Big\|\F^{-1}\Big[\F\big[\tfrac{1}{ix}\zeta^s(x)\big](\phi_\nu^{-1})'(-\cdot)\Big](x)\Big\|_\infty\\
    \lesssim&\big\|\tfrac{1}{x}\zeta^s(x)\big\|_{H^{\beta}}+\big\|\zeta^s\big\|_{H^{\beta}}.
  \end{align*}
  Combining with \eqref{eqContPart}, we get \eqref{eqBoundNorm}.

  \textit{Step 3:}
  Define the sets
  \begin{align*}
    \mathcal G':= C^\infty(\R)\cap L_0^2(P_\nu) \quad\text{and}\quad
    \mathcal H':= \big\{b\in C^\infty(\R)|b(0)=0\big\}\cap \dot\Theta_\nu
  \end{align*}
  which are larger than $\mathcal G$ an $\mathcal H$ from above. Using the Fourier multiplier property on Besov spaces \eqref{eqFourMult} and an analogous result for the Fourier multiplier $\phi_\nu(-\cdot)$, we obtain 
  \[
      \|P_\nu(-\cdot)\ast f\|_{C^{s'}}\lesssim\|f\|_{C^{s'}}\quad\text{and}\quad
      \|F^{-1}[\phi_\nu^{-1}(-\cdot)\F f]\|_{C^{s}}\lesssim\|f\|_{C^{s'}}
  \]
  for any $s>0$ and $f\in C^{s'}$ for $s'>s+\beta$. Therefore, following the lines of the proof of Lemma~\ref{lemAdjoint}(i), we see that $A_\nu^\star|_{\mathcal G'}$ is given by $A_\nu^\star g=P_\nu(-\cdot)\ast g$ for $g\in\mathcal G'$ and that $A_\nu^\star|_{\mathcal G'}$ is a bijection from $\mathcal G'$ onto $\mathcal H'$ with inverse $(A_\nu^\star|_{\mathcal G'})^{-1} b=F^{-1}[\phi_\nu^{-1}(-\cdot)\F b]$ for $b\in\mathcal H'$.

  By Proposition~\ref{propApprox} for any $\zeta$ a necessary and sufficient condition to be in the range of $A_\nu^\star$ is the existence of a sequence $(\chi_m)_{m\in\N}\subset\mathcal H'$ such that $\chi_m\to\zeta$ in $L^2(\nu)$ and $(A_\nu^\star)^{-1}\chi_m$ converges in $L^2(P_\nu)$. Now, for any $\zeta\in Z^\beta\cap L^1(\nu)\cap L^2(\nu)$ we find $\chi_m=\chi_m^s+\chi_m^c$ with $\chi_n^s,\chi_n^c\in\mathcal H'$ satisfying $\chi_m^s\to\zeta^s$ and $\chi_m^c\to\zeta^c$ in $L^2(\nu)$ as well as
  \begin{align*}
    \Big\|\F^{-1}\Big[\frac{\F[\zeta-\chi_n]}{\phi_\nu(-\cdot)}\Big]\Big\|_{L^2(P_\nu)}\lesssim& \big\|(\zeta^s-\chi_n^s)(x)\big\|_{H^{\beta}}\\
      &+\big\|\tfrac1x(\zeta^s-\chi_n^s)(x)\big\|_{H^{\beta}}+\|\zeta^c-\chi^c_n\|_{C^{\beta+\eps}}\to0
  \end{align*}
  for $m\to\infty$ owing to \eqref{eqBoundNorm}. Hence, $\zeta\in\ran A_\nu^\star$ with $(A_\nu^\star)^{-1}\zeta=\F^{-1}[\phi_\nu^{-1}(-\cdot)\F\zeta]=\F^{-1}[\phi_\nu^{-1}(-\cdot)]\ast\zeta$.
  \qed

\bibliography{bib} 

\end{document}